\documentclass[11pt,letter]{article} 
\usepackage[papersize={8.5in,11in},margin=1in]{geometry}
\usepackage[format=hang,font={small}]{caption}

\usepackage{times}
\usepackage{microtype}
\usepackage{graphicx}
\usepackage{booktabs} 

\usepackage{hyperref}

\usepackage{url}            
\usepackage{amsfonts}       
\usepackage{nicefrac}       
\usepackage{microtype}      

\usepackage{multirow}

\usepackage{times}
\usepackage{algorithm}
\usepackage{xspace}
\usepackage[numbers,sort,compress]{natbib} 

\usepackage[noblocks]{authblk}

\usepackage{graphicx}
\usepackage{amsmath}
\usepackage{amssymb}
\usepackage{bbding}
\usepackage{pifont}
\usepackage{wasysym}
\usepackage{bm}
\usepackage{algorithmic}
\usepackage{epstopdf}
\usepackage{psfrag}
\usepackage{bbm}
\usepackage{enumerate}
\usepackage{subeqnarray}
\usepackage{cases}

\AtBeginDocument{%
 \abovedisplayskip=6pt minus 1pt
 \abovedisplayshortskip=4pt plus 1pt
 \belowdisplayskip=6pt minus 1pt
 \belowdisplayshortskip=4pt plus 1pt
}

\def\circlenum#1{\text{\textcircled{\textbf{#1}}}}

\newcommand{\E}{\mathbb E}

\newcommand{\defi}{\text{def}}

\newcommand{\argmin}{\mathop{\rm arg\min}}

\newcommand{\bbR}{\mathbb{R}}

\newcommand{\w}{{\mathbf{w}}}

\renewcommand{\O}{{\mathcal{O}}}

\newcommand{\tO}{\tilde{\mathcal{O}}}

\usepackage{color}

\usepackage{soul}

\newcommand{\beq}{\begin{equation}}
\newcommand{\eeq}{\end{equation}}
\newcommand{\beqa}{\begin{eqnarray}}
\newcommand{\eeqa}{\end{eqnarray}}
\newcommand{\beqas}{\begin{eqnarray*}}
\newcommand{\eeqas}{\end{eqnarray*}}
\newcommand{\bi}{\begin{itemize}}
\newcommand{\ei}{\end{itemize}}
\newcommand{\ba}{\begin{array}}
\newcommand{\ea}{\end{array}}

\def\argmin{{\rm argmin}}

\def\vgap{\vspace*{.1in}}

\def\exp{{\rm exp}}

\newcommand{\bbr}{\Bbb{R}}
\def\w{\omega}

\usepackage{enumitem}
\setenumerate[1]{itemsep=0pt,partopsep=0pt,parsep=\parskip,topsep=2pt}
\setitemize[1]{itemsep=0pt,partopsep=0pt,parsep=\parskip,topsep=2pt}
\setdescription{itemsep=0pt,partopsep=0pt,parsep=\parskip,topsep=2pt}
\newtheorem{theorem}{Theorem}
\newtheorem{lemma}{Lemma}
\newtheorem{corollary}{Corollary}

\newtheorem{assumption}{Assumption}

\newtheorem{definition}{Definition}

\author[*]{Chaobing Song}
\author[+]{Ji Liu}
\author[$\ddagger$]{Yong Jiang}
\affil[*]{Tsinghua University, $\quad\quad$\texttt{\href{mailto:songcb16@mails.tsinghua.edu.cn}{\color{black}songcb16@mails.tsinghua.edu.cn}}}
\affil[+]{Rochester University, $\quad\quad$   \texttt{\href{mailto:ji.liu.uwisc@gmail.com}{\color{black}ji.liu.uwisc@gmail.com}}}
\affil[$\ddagger$]{Tsinghua University, $\quad\quad$   \texttt{\href{mailto:jiangy@mail.sz.tsinghua.edu.cn}{\color{black}jiangy@mail.sz.tsinghua.edu.cn}}}

\date{Feb 25, 2019}

\begin{document}

\title{Inexact Proximal Cubic Regularized Newton Methods for Convex Optimization}

  \maketitle

\begin{abstract}
In this paper, we use Proximal Cubic regularized Newton Methods (PCNM) to optimize  the sum of a smooth convex function and a non-smooth convex function, where we use inexact gradient and Hessian, and an inexact subsolver for the cubic regularized second-order subproblem. We propose inexact variants of PCNM and accelerated PCNM respectively, and show that both variants can achieve the same convergence rate as in the exact case, provided that the  errors in the inexact gradient, Hessian and subsolver decrease at appropriate rates. Meanwhile, in the online stochastic setting where data comes endlessly, we give the overall complexity of the proposed algorithms and show that they are as competitive as the stochastic gradient descent. Moreover, we give the overall complexity of the proposed algorithms in the finite-sum setting and show that it is as competitive as the state of the art variance reduced algorithms. Finally, we propose an efficient algorithm for the cubic regularized second-order subproblem, which can converge to an enough small neighborhood of the optimal solution in a superlinear rate.



\end{abstract}

\section{Introduction}
In this paper, we study the following composite optimization problem  
\begin{equation}
\min_{x\in\bbR^d}\left\{ F(x) \overset{\defi}{=}f(x)+h(x)\right\}, \label{eq: prob}
\end{equation}
where $f(x)$ is twice differentiable and convex, and $h(x)$ is simple, nonsmooth and convex. To solve \eqref{eq: prob}, the popular methods are first-order algorithms such as Proximal Gradient Method (PGM) and Accelerated Proximal Gradient Method (APGM) \cite{nesterov2007gradient}. Assume $x^* \overset{\defi}{=}\argmin_{x\in\bbR^d} F(x)$. In order to find an $\epsilon$-accurate solution $x$ such that $F(x)-F(x^*)\le \epsilon$, PGM needs $\mathcal{O}(\epsilon^{-1})$ iterations and  APGM need $\mathcal{O}(\epsilon^{-1/2})$ iterations, where  $\mathcal{O}(\epsilon^{-1/2})$ is the optimal rate for first-order methods. However, if we use the second-order information of $f(x)$, from \cite{nesterov2006cubic,nesterov2008accelerating,nunes2018accelerated}, we know that the Proximal Cubic regularized Newton Method (PCNM) with the following iterative procedure
\begin{equation}
x_{t+1}\overset{\defi}{=}\argmin_{x\in\bbR^d}\Big\{ f(x_t)+\langle \nabla f(x_t), x-x_t\rangle+\frac{1}{2}\langle \nabla^2 f(x_t)(x-x_t), x-x_t\rangle+\frac{\eta}{6}\|x-x_t\|^3+ h(x)\Big\}, \label{eq:cnm}
\end{equation}  
needs $\mathcal{O}(\epsilon^{-1/2})$ iterations to find an $\epsilon$-accurate solution, where $\nabla f(x_t)$ denotes the gradient at $x_t$, $\nabla^2 f(x_t)$ denotes the Hessian matrix at $x_t$ and $\eta$ is a parameter to be determined. Meanwhile the  accelerated PCNM (APCNM) only need $\mathcal{O}(\epsilon^{-1/3})$ iterations. Both iteration complexity results are better than that of PGM and APGM respectively. 

Although the iteration complexities of PCNM and APCNM are superior, unlike PGM and APGM, they are seldom used in the large-scale optimization where the problem size ($i.e.,$ the number of data samples $n$ and the dimension $d$) is often large. This is because when the problem size is large, computing the exact gradient and Hessian is very expensive; meanwhile solving the subproblem \eqref{eq:cnm} exactly needs matrix factorization or inversion which scales poorly with the problem size. 
As a result, although the iteration complexities are better, as the problem size becomes large, the overall complexity of PCNM and APCNM will not be as competitive as PGM and APGM respectively. In fact,  
 it is commonly believed that second-order methods such as PCNM and APCNM are only suitable for small-scale problems, and first-order methods such as PGM and APGM are superior in the large-scale setting.

In this paper, we consider the inexact variants of PCNM and APCNM and show the following result from a theoretical view: the overall complexity of second-order methods can be as competitive as that of first order algorithms in the large-scale setting or even in the online setting where data arrives endlessly. In fact, in the strongly convex setting, the overall complexity of APCNM will have a better dependence on the strong convexity constant than the state of the art stochastic gradient descent (SGD) algorithm. 
We obtain the competitive overall complexity results by using the inexact gradient and Hessian and finding an approximate solution of the subproblem \eqref{eq:cnm} with properly decreased errors. The proposed results implies that \emph{the order of the information we use may be not the factor that determines the scalability of algorithm, if we tune the factors in the corresponding subproblem (such as \eqref{eq:cnm} ) properly}.

In Section \ref{sec:result}, we review the related work and give the main results from four aspects:
\begin{itemize}
\item the research progress of  the inexact variants of Cubic regularized Newton Method (CNM);
\item  the overall complexity in the online setting;
\item the overall complexity in the finite-sum setting;
\item the proposed efficient subsolver for  \eqref{eq:cnm}. 
\end{itemize}
Then in Section \ref{sec:cubic}, we propose the Inexact Proximal Cubic regularized Newton Method (IPCNM) and give its theoretical analysis; in  Section \ref{sec:acc-cubic}, we propose Accelerated Inexact Proximal Cubic regularized Newton Method (AIPCNM) and  give its theoretical analysis; in  Section \ref{sec:cubic-svrg}, we propose the efficient Cubic Proximal Stochastic Variance Reduced Gradient (Cubic-Prox-SVRG) method and show that it can converge to a neighborhood of the optimal solution in a superlinear rate.




\section{Related Work and Main Results}\label{sec:result}
Before continue, we provide the notations and problem setting first. 
Let $I$ as the identity matrix with a proper size according to the context. 
Let $x^*$ as a minimizer of $F(x)$ and we say $x$ is an $\epsilon$-accurate solution if it satisfies $F(x)-F(x^*)\le \epsilon$. 
Throughout this paper, we use $\|\cdot\|$ to denote the Euclidean norm of a vector or the spectral norm of a matrix. Denote $\nabla F(x)\overset{\defi}{=}\nabla f(x)+h^{\prime}(x)$, where $h^{\prime}(x)\in \partial h(x)$. 
We use the big  $O$ notation $\mathcal{O}(\cdot)$ to denote the computational complexity and  $\tilde{\mathcal{O}}(\cdot)$ denote the complexity result that hide the ploy logarithmic terms.

\begin{definition}
For a function $F:\bbR^d\rightarrow \bbR$, $F(x)$ is $\sigma_2$-strongly convex if $\forall x, y \in\bbR^d$, $$F(y)\ge F(x)+\langle\nabla F(x), y-x\rangle+ \frac{\sigma_2}{2}\|y-x\|^2;$$
if $\sigma_2=0$, then $F(x)$ is only convex.
\end{definition}

\begin{definition}
For a function $f:\bbR^d\rightarrow \bbR$, $f(x)$ has $L_2$-Lipschitz gradients if $\;\forall x, y \in\bbR^d$, $$\|\nabla f(x) - \nabla f(y)\|\le L_2 \|x-y\|;$$  $f(x)$ has $L_3$-Lipschitz Hessians if $\forall x, y \in\bbR^d$,   $$\|\nabla^2 f(x) - \nabla^2 f(y)\|\le L_3 \|x-y\|.$$
\end{definition}

When the nonsmooth term $h(x)$ exists, $f(x)$ has $L_3$-Lipschitz Hessian and $F(x)$ is convex, by \cite{nunes2018accelerated}, APCNM can find an $\epsilon$-accurate solution with $\O(\epsilon^{-1/3})$ iterations.  Meanwhile by extending the result  of the smooth setting trivially \cite{nesterov2006cubic}, we can know that PCNM can find an $\epsilon$-accurate solution for \eqref{eq: prob} with $\O(\epsilon^{-1/2})$ iterations. Except \cite{nunes2018accelerated}, the existing researches of inexact variants of CNM mainly focus on the smooth setting where the nonsmooth term $h(x)$ does not exist. 

In the nonconvex setting,
 in order to reduce the high computational cost in optimizing the subproblem \eqref{eq:cnm} and maintain the convergence rate of the exact case at the same time, \cite{cartis2011adaptive,cartis2011adaptive2,kohler2017sub} considered a subsampling strategy to obtain inexact gradient and Hessian, and a termination condition for optimizing \eqref{eq:cnm} , while the conditions of subsampling depend on the further iteration and thus is implementable, and the termination condition is specific to the Lanzcos method \cite{carmon2018analysis}. \cite{zhou2018stochastic} used variance reduction strategy to reduce the complexity of computing the  gradient and Hessian, while the complexity to update the variance reduced gradient and Hessian is $O(d^2)$ and thus the SVRC method in \cite{zhou2018stochastic} is only suitable for the problem with small dimension $d$. \cite{tripuraneni2018stochastic} made a considerable progress that the stochastic cubic regularization method in \cite{tripuraneni2018stochastic} needs $\tO(\epsilon^{-3.5})$ stochastic gradient and stochastic Hessian-vector product evaluations to find an approximate local minima for general smooth, nonconvex functions, which matches the best known result,  while it did not give the analysis of the convex setting.  
\begin{table*}[ht]
\caption{Comparison of inexact cubic regularized Newton methods in the convex setting}
\label{tb:1}
\vskip 0.15in
\begin{center}
\begin{small}
\begin{tabular}{ccccccccr}
\toprule
  Cubic                       & Same             & Inexact  & Inexact    & Inexact          & Nonsmooth      \\
   Method                      & Rate?            & Hessian? & Gradient?  & Subsolver?      & Regularizer?     \\
\midrule
\cite{cartis2012evaluation} & {\large $\times$} 				& \Checkmark      & {\large $\times$}         & \Checkmark  			 &	{\large $\times$}		 \\
\cite{ghadimi2017second}    & {\large $\times$} 			& \Checkmark 	   & {\large $\times$} 	    & {\large $\times$} 					&{\large $\times$}	  \\
\cite{chen2018adaptive}    & \Checkmark  			& \Checkmark 	   & {\large $\times$} 	    & \Checkmark 			 &	{\large $\times$}		  \\
\textbf{This Paper}    				&\Checkmark	  		& \Checkmark 	   & \Checkmark 	    & \Checkmark 		  	&	\Checkmark 	  \\
\bottomrule
\end{tabular}
\end{small}
\end{center}
\vskip -0.1in
\end{table*}

 Compared with finding an $\epsilon$-second order stationary point of the nonconvex setting, in the convex setting, the goal of finding an $\epsilon$-accurate solution in terms of objective function results in extra difficulty. In this paper, we propose the inexact PCNM (IPCNM) and accelerated IPCNM (AIPCNM) as the inexact invariants of PCNM and APCNM respectively.
Table \ref{tb:1} gives the researches about inexact variants of CNM in the convex setting. As shown in Table \ref{tb:1}, only the algorithms in \cite{chen2018adaptive} and this paper can maintain the same convergence rate as the exact case; all the researches consider using inexact Hessian, while only this paper also use inexact gradient; \cite{cartis2011adaptive,chen2018adaptive} and this paper use inexact subsolver, while only the termination condition for subsolver in this paper is not specific to the Lanzcos method; finally, only the results in this paper are applicable to the case where the nonsmooth regularizer $h(x)$ exists.

In the following discussion, the cubic regularized second-order approximation function we need to optimize in each iteration is defined as
 \begin{eqnarray}
\tilde{f}_{\eta}(x;y)&\overset{\defi}{=}&f(y) +\langle g, x-y\rangle + \frac{1}{2}\langle H(x-y), x-x_t\rangle+\frac{\eta}{6}\|x-y\|^3 + h(x),\label{eq:subprob}
\end{eqnarray}
where $g$ is an inexact gradient on $y$,  $H_t\succeq 0$ is an inexact Hessian on $y$, $\eta>0$ is a parameter to be determined. Then we make Assumption \ref{ass:appro}.

\begin{assumption}\label{ass:appro}
Assume $\tilde{x}\overset{\rm def}{=}\argmin_{x\in\bbR^d}\tilde{f}_{\eta}(x;y)$.  
The subsolver we use  can find an $\epsilon$-accurate solution $z$ such that $\tilde{f}_{\eta}(z;y) - \tilde{f}_{\eta}(\tilde{x};y)\le \epsilon$ with at most $\O(\text{\rm cost}(H_tv)\log\frac{1}{\epsilon})$ cost, where $\text{\rm cost}(H_tv)$ denotes the cost of Hessian-vector products. 
\end{assumption}
From the convergence analysis \cite{carmon2018analysis}, it is known that the Lanzcos method can satisfy Assumption \ref{ass:appro} when the nonsmooth term $h(x)$ does not exist. In Section \ref{sec:cubic-svrg}, we propose the Cubic-Prox-SVRG method and show that it can converge to an enough small neighborhood of $\tilde{x}$ in a superlinear rate.

\subsection{Overall complexity in the online stochastic setting}
In the online stochastic setting where  data arrives sequentially and endlessly, $f(x)$ can be written as an expectation of the stochastic function $f(x;\epsilon)$, then the problem \eqref{eq: prob} is 
\begin{eqnarray}
\min_{x\in\bbR^d}\left\{ F(x) \overset{\defi}{=}f(x)+h(x)\overset{\defi}{=}\E_{\xi\sim \mathcal{D}}[f(x;\xi)]+h(x)\right\}, \label{eq: online-prob}
\end{eqnarray}
where $\xi$ is a random variable sampled from an underlying distribution $\mathcal{D}$. Meanwhile, we make two general assumptions \ref{ass:gradient-Hessian} and \ref{ass:g-H}. 
\begin{assumption}\label{ass:gradient-Hessian}
$\nabla f(x;\xi)$ satisfies $\forall x\in \bbR^d$, $$\E[\nabla f(x;\xi)] = \nabla f(x),\quad\quad  \E[\|\nabla f(x;\xi)-\nabla f(x)\|^2]\le \tau_1^2,\quad\quad\|\nabla f(x;\xi)-\nabla f(x)\|\le \gamma_1$$ almost surely; meanwhile,
$\nabla^2 f(x;\xi)$ satisfies  $\forall x\in \bbR^d$, $$\E[\nabla^2 f(x;\xi)] = \nabla^2 f(x),\quad\quad \| \E[(\nabla^2 f(x;\xi)-\nabla^2 f(x))^2]\|\le \tau_2^2,\quad\quad\|\nabla^2 f(x;\xi)-\nabla^2 f(x)\|\le \gamma_2$$ almost surely; moreover the cost of the stochastic Hessian-vector product $\nabla^2 f(x;\xi)v$ is not higher than the vector-vector inner product $\langle \nabla f(x;\xi), v\rangle$. 
\end{assumption}

\begin{assumption}\label{ass:g-H}
In the $t$-th iteration of IPCNM, we set $$g_t\overset{\rm def}{=}\frac{1}{\hat{n}_{t1}}\sum_{i=1}^{\hat{n}_{t1}} \nabla f(x;\xi),\quad\quad H_t\overset{\rm def}{=}\frac{1}{\hat{n}_{t2}}\sum_{i=1}^{\hat{n}_{t2}} \nabla^2 f(x;\xi),$$
where $\hat{n}_{t1}, \hat{n}_{t2}$ are the number of stochastic gradient samples and stochastic Hessian samples to be determined. Meanwhile, 
in the $t$-th iteration of AIPCNM, we set $$g_t\overset{\rm def}{=}\frac{1}{\bar{n}_{t1}}\sum_{i=1}^{\bar{n}_{t1}} \nabla f(x;\xi), \quad\quad  H_t\overset{\rm def}{=}\frac{1}{\bar{n}_{t2}}\sum_{i=1}^{\bar{n}_{t1}} \nabla^2 f(x;\xi)+\mu_t I,$$
where $\bar{n}_{t1}, \bar{n}_{t2}$ are the number of stochastic gradient samples and stochastic Hessian samples to be determined, $\mu_t>0$ is a parameter to be determined.
\end{assumption}
By Assumption \ref{ass:g-H}, from the $0$-th iteration to $t$-th iteration, in IPCNM, the number of stochastic gradient samples is $\sum_{i=1}^t \hat{n}_{i1}$, and the number of stochastic Hessian samples is $\sum_{i=1}^t \hat{n}_{i2}$; in AIPCNM, they are $\sum_{i=1}^t \bar{n}_{i1}$ and $\sum_{i=1}^t \bar{n}_{i2}$ 
respectively. Based on Assumption \ref{ass:appro}, in IPCNM, the total complexity of calling the subsolver will be 
at most $\O(\sum_{i=0}^{t} \hat{n}_{i2}\log\frac{1}{\epsilon_i})$ stochastic Hessian-vector products, where $\epsilon_i$ is the accuracy we need to attain in the solving procedure of the subproblem $\argmin_{x\in\bbR^d} \tilde{f}_{\eta}(x;y)$  of the $i$-th iteration. Correspondingly, in AIPCNM, the total cost of calling the subsolver will be  $\O(\sum_{i=0}^{t} \bar{n}_{i2}\log\frac{1}{\epsilon_i})$ stochastic Hessian-vector products. By Assumption \ref{ass:gradient-Hessian}, the cost of  stochastic Hessian-vector products is not higher than that of stochastic  gradient evaluation. Therefore we measure the overall complexity by the number of equivalent stochastic  gradient evaluations.

Table \ref{tb:2} gives the overall complexity of representative algorithms in the online stochastic setting. For simplicity, in Table \ref{tb:2}, we neglect the poly-logarithmic factor and use the $\tilde{O}$ notation. 
The existing algorithms under this setting are mainly first-order algorithms \cite{shalev2012online}, which  can be divided into methods that pass one sample or a fixed mini-batch samples in each iteration \cite{duchi2010composite,xiao2010dual}, and methods that use an increased sample size in each iteration \cite{byrd2012sample,friedlander2012hybrid,schmidt2011convergence}. If we do not consider the poly-logarithmic factor, COMID \cite{duchi2010composite} obtains the optimal convergence rate ($i.e,$ regret)  $\tO(\epsilon^{-2})$ in the convex setting and $\tO(\epsilon^{-1})$ in the $\sigma_2$-strongly convex setting. 
However, as shown in Table \ref{tb:2},  the Inexact Proximal Gradient  Method (IPGM) and Accelerated IPGM (IPGM) \cite{schmidt2011convergence}{\footnote{Although \cite{schmidt2011convergence} do not give the  overall complexity in the online stochastic setting, we can give the complexity results in Table \ref{tb:2} based on the same analysis in this paper.}}, which belongs to the methods with increased sample size, can not obtain the optimal rate in the convex setting. The proposed methods IPCNM and AIPCNM belongs to the methods with an increased sample size, while we use the second-order information. 

In Table \ref{tb:2}, it is shown that IPCNM and AIPCNM have better overall complexity than the corresponding IPGM and AIPGM respectively in the convex setting. AIPCNM can obtain the optimal rate in both convex and strongly convex setting. Particularly, AIPCNM has better dependence on the strong convexity constant $\sigma_2$ than COMID.


\begin{table*}[th]
\caption{Overall complexity of algorithms in the online stochastic setting}
\label{tb:2}
\vskip 0.15in
\begin{center}
\begin{small}
\begin{tabular}{c|ccccr}
\toprule
Setting  &  Method                     &  $\#$-equivalent stochastic gradient evaluations                                \\
\midrule
Convex&COMID \cite{duchi2010composite} &  $\tilde{\mathcal{O}}(\epsilon^{-2})$ 			 \\
&IPGM \cite{schmidt2011convergence}     & $\tilde{\mathcal{O}}(\epsilon^{-3}) $				\\
&AIPGM \cite{schmidt2011convergence}    & $\tilde{\mathcal{O}}(\epsilon^{-5/2})$				\\
&IPCNM (\textbf{this paper})    & $\tilde{\mathcal{O}}(\epsilon^{-5/2})$ 				  \\
&AIPCNM (\textbf{this paper})    & $\tilde{\mathcal{O}}(\epsilon^{-2})$  	 \\
\midrule
Strongly &COMID \cite{duchi2010composite} &  $\tilde{\mathcal{O}}(\sigma_2^{-1}\epsilon^{-1})$ 			 \\
Convex&IPGM \cite{schmidt2011convergence}    & $\tilde{\mathcal{O}}(\sigma_2^{-1}\epsilon^{-1})$					\\
&IAPGM \cite{schmidt2011convergence}    & $\tilde{\mathcal{O}}(\sigma_2^{-3/2}\epsilon^{-1})$	   			\\
&IPCNM (\textbf{this paper})    & $\tilde{\mathcal{O}}(\sigma_2^{-5/6}\epsilon^{-4/3})$ 			  \\
&AIPCNM (\textbf{this paper})    & $\tilde{\mathcal{O}}(\sigma_2^{-2/3}\epsilon^{-1})$   \\

\bottomrule
\end{tabular}
\end{small}
\end{center}
\vskip -0.1in
\end{table*}

\subsection{Overall complexity in the finite-sum setting}
In the finite-sum setting where $f(x)$ has a finite-sum structure, the problem \eqref{eq: prob} can be written as 
\begin{eqnarray}
\min_{x\in\bbR^d}\left\{ F(x) \overset{\defi}{=}f(x)+h(x)\overset{\defi}{=}\frac{1}{n}\sum_{i=1}^n f_i(x)+h(x)\right\}. \label{eq: online-prob}
\end{eqnarray}
In the finite-sum setting, we also assume that Assumptions \ref{ass:gradient-Hessian} and \ref{ass:g-H} hold. 
Because if the number of samples is $n$, then we can obtain the exact gradient or Hessian, in the finite-sum setting, the number of samples for gradient and Hessian is at most $n$.

To solve \eqref{eq: online-prob}, the state of the art algorithms are based on the well-known variance reduction technique \cite{johnson2013accelerating,xiao2014proximal,allen2017katyusha}. In Table \ref{tb:2}, we use SVRG and Katyusha as the representative algorithms of non-accelerated variance reduced method and accelerated variance reduced method respectively. 

As shown in Table \ref{tb:2}, an important advantage of AIPCNM and AIPCNM is they do not need to pass all the data if we only want a low-accurate solution. Meanwhile, in the convex setting, the AIPCNM method has a faster rate $\tO(\epsilon^{-1/3})$, therefore it can obtain a high-accuracy solution faster than optimal gradient method Katyusha. In the strongly convex setting, the AIPCNM method has a better dependence on the strong convexity constant $\sigma_2$.

\begin{table*}[th]
\caption{Comparision of representative algorithms in the finite-sum setting}
\label{tb:2}
\vskip 0.15in
\begin{center}
\begin{small}
\begin{tabular}{c|ccccr}
\toprule
Setting  &  Method                     &  $\#$-equivalent stochastic gradient evaluations                   \\
\midrule
Convex &SVRG \cite{johnson2013accelerating}     & $\tilde{\mathcal{O}}(n\epsilon^{-1}) $			\\
&Katyusha \cite{allen2017katyusha}   & $\tilde{\mathcal{O}}(n\epsilon^{-1/2})$	   			\\
&IPCNM (\textbf{this paper})    & $\min\{\tilde{\mathcal{O}}(\epsilon^{-5/2}),\tilde{\mathcal{O}}({n\epsilon^{-1/2}})\} $ 			  \\
&AIPCNM (\textbf{this paper})    & $\min\{\tilde{\mathcal{O}}(\epsilon^{-2}), \tilde{\mathcal{O}}(n\epsilon^{-1/3})\}$  		\\
\midrule
Strongly &SVRG \cite{johnson2013accelerating}    & $\tilde{\mathcal{O}}(n+n\sigma_2^{-1})$				\\
Convex&Katyusha \cite{allen2017katyusha}    & $\tilde{\mathcal{O}}(n+n^{1/2}\sigma_2^{-1/2})$   			\\
&IPCNM (\textbf{this paper})    & $\min\left\{\tilde{\mathcal{O}}(\sigma_2^{-5/6}\epsilon^{-4/3}), \tilde{\mathcal{O}}(\sigma_2^{-1/2}n)\right\}$ 					  \\
&AIPCNM (\textbf{this paper})    & $\min\left\{\tilde{\mathcal{O}}(\sigma_2^{-2/3}\epsilon^{-1}),\tilde{\mathcal{O}}(\sigma_2^{-1/3}n)\right\} $  \\

\bottomrule
\end{tabular}
\end{small}
\end{center}
\vskip -0.1in
\end{table*}
\subsection{Efficient subsolver for the cubic regularized second-order subproblem}
In Section \ref{sec:cubic-svrg}, we propose the Cubic-Prox-SVRG method to solve the subproblem $\min_{x\in\bbR^d}\tilde{f}_{\eta}(x;y)$ by exploring the finite-sum structure in the inexact Hessian and the uniform convexity of the cubic regularizer $\frac{1}{3}\|\cdot\|^3.$ Because it can converge to an enough small neighborhood of the optimal solution in a superlinear rate (``enough'' means the approximate solution satisfies the need of IPCNM and AIPCNM), in the convex setting, it is a good alternative to the well-known Lanczos method which only has a linear rate \cite{carmon2018analysis}.
\section{The inexact proximal cubic regularized Newton method}\label{sec:cubic}

\begin{algorithm}[htb]
   \caption{Inexact proximal cubic regularized Newton method (IPCNM)}
   \label{alg:icnm}
\begin{algorithmic}[1]
	\STATE Input: $x_0\in\bbR^d$ and  $\eta = 3L_3$
    \FOR{t=0,1,2,...}
    \STATE Compute an inexact gradient $g_t$ and an inexact Hessian $H_t$ on $x_t$ 
    \STATE Compute an inexact solution $x_{t+1}$ of the subproblem $\min_{x\in\bbR^d}\tilde{f}_{\eta}(x;x_t)$
    \ENDFOR
\end{algorithmic}\label{alg:icnm}
\end{algorithm}

In this section, we propose  IPCNM  in Algorithm \ref{alg:icnm}. For Algorithm \ref{alg:icnm}, we have Lemma \ref{lem:basic-f} which is key to the convergence result.
\begin{lemma}\label{lem:basic-f}
Let $\{x_t\}_{t\ge0}$ be generated by Algorithm \ref{alg:icnm} and $x_{t+1}^* = \argmin_{x\in \bbR^d}\tilde{f}_{\eta_t}(x; x_t)$. 
Define the approximation error 
\begin{equation}
E_t \overset{\rm def}{=} \frac{4}{3L_3^2}\|\nabla^2 f(x_t) - H_t\|^3
+ \frac{4}{3}\left(\frac{2}{L_3}\right)^{1/2}\|\nabla f(x_t) - g_t\|^{3/2} + (\tilde{f}_{\eta_t}(x_{t+1}; x_t) -\tilde{f}_{\eta_t}(x_{t+1}^*; x_t)).\label{eq:E-t}
\end{equation}
Then $\forall 0\le \alpha_t\le 1$, if $F(x)$ is convex and has $L_3$-Lipschitz Hessians,  it follows that for $t\ge1$
\begin{equation}
F(x_{t})\le F(\alpha_{t-1}x^* + (1-\alpha_{t-1})x_t)+L_3\alpha_{t-1}^3\|x_{t-1}-x^*\|^3+E_{t-1}. \label{eq:meta}
\end{equation}
\end{lemma}
In Lemma \ref{lem:basic-f},  $E_t$ is the approximate error by the inexact gradient $g_t$, the inexact Hessian $H_t$ and the inexact solution $x_{t+1}$ of the subproblem \eqref{eq:subprob}.
To continue, we give the general assumption that the Euclidean distance between the iterates $\{x_t\}_{t\ge0}$ and $x^*$ is bounded by a constant.  
\begin{assumption}\label{ass:x}
Let $\{x_t\}_{t\ge0}$ be generated by Alg. \ref{alg:icnm}. Then, there exists $D>0$ such that $\forall t\ge0$, $\|x_t-x^*\|\le D$.
\end{assumption}

Then based on Lemma \ref{lem:basic-f} and Assumption \ref{ass:x} 
, Theorems \ref{thm:nonstrong} and \ref{thm:strong}  gives the convergence result in the  convex setting and strongly convex setting respectively.
\begin{theorem}[The convex setting]\label{thm:nonstrong}
Suppose that Assumption \ref{ass:x} holds and $E_t$ is defined in Lemma \ref{lem:basic-f}. For the convex function $F(x)$ in \eqref{eq: prob} 
, then it follows that for $t\ge 1$, 
\begin{eqnarray}
F(x_{t}) - F(x^*)\le\frac{27L_3D^3}{(t+1)(t+2)} + \frac{1}{t(t+1)(t+2)}  \sum_{i=1}^{t} i(i+1)(i+2)E_i.\nonumber\label{eq:nonstrong-2}
\end{eqnarray}
\end{theorem}
By Theorem \ref{thm:nonstrong}, if $E_i=\O\left(\frac{1}{(i+2)^3}\right)$, then IPCNM can converge to an $\O(1/t^2)$-accurate solution after the $t$-th iteration. 
\begin{theorem}[The strongly convex setting]\label{thm:strong}
Suppose that Assumption \ref{ass:x} hold, $E_t$ is defined in Lemma \ref{lem:basic-f}. Assume for $t\ge 0$, 
\begin{equation}
\alpha =  \min\left\{\frac{1}{3}, \sqrt{\frac{\sigma_2}{3L_3D}}\right\}.
\end{equation}
Then for the $\sigma_2$-strongly convex function $F(x)$, it follows that for $t\ge 1$, 
\begin{eqnarray}
F(x_{t}) - F(x^*)\le  (1-\alpha)^t\left(F(x_{0}) - F(x^*)\right) +(1-\alpha)^t \sum_{i=0}^{t-1}\frac{E_i}{(1-\alpha)^i}.\nonumber
\end{eqnarray}
\end{theorem}
By Theorem \ref{thm:strong},  if $E_i=\O\left(\frac{(1-\alpha)^i}{t}\right)$, then IPCNM can converge to an $\O((1-\alpha)^t)$-accurate solution after the $t$-th iteration. In Theorem \ref{thm:superlinear}, we also show that IPCNM has a local superlinear  rate.
\begin{theorem}[Local superlinear convergence rate]\label{thm:superlinear}
Define $\omega \overset{\rm def}{=} \frac{1}{L_3^2}\left(\frac{\sigma_2}{2}\right)^3$.
Assume that $t_0$ is the minimal integer such that $F(x_{t_0}) - F(x^*)\le \frac{2}{3}\omega$.  Then for $t\ge t_0$, by setting $$E_t\le\frac{\omega}{2}\left(2/3\right)^{(3/2)^{t-t_0+1}},$$ we have for $t\ge t_0$, 
\begin{eqnarray}
F(x_{t})-F(x^*)\le \omega(2/3)^{(3/2)^{t-t_0}}.\nonumber
\end{eqnarray}
\end{theorem}

Finally, we give Corollaries  \ref{thm:acc-online} and \ref{thm:acc-finite} to show the overall complexity in the online stochastic setting. 
\begin{corollary}[The online stochastic setting]\label{thm:online}
Suppose that Assumptions \ref{ass:gradient-Hessian} and \ref{ass:g-H} hold. Then if $F(x)$ is convex, then with the probability $1-\delta$, IPCNM can find an $\epsilon$-accurate solution in at most  $$\tO(\epsilon^{-5/2}) \text{  {\rm equivalent stochastic gradient iterations.}} $$

If $F(x)$ is $\sigma_2$-strongly convex, IPCNM can find an  $\epsilon$-accurate solution in at most  $$\tO(\sigma^{-5/6}\epsilon^{-4/3}) \text{  {\rm equivalent stochastic gradient iterations.}}$$
\end{corollary}

\begin{corollary}[The finite-sum setting]\label{thm:finite}
Suppose that Assumptions \ref{ass:gradient-Hessian} and \ref{ass:g-H} hold. Then if $F(x)$ is convex, then with the probability $1-\delta$, IPCNM can find an $\epsilon$-accurate solution in at most  $$\min\{\tilde{\mathcal{O}}(\epsilon^{-5/2}),\tilde{\mathcal{O}}({n\epsilon^{-1/2}})\}
  \text{  {\rm equivalent stochastic gradient iterations.}}$$ 

If $F(x)$ is $\sigma_2$-strongly convex, IPCNM can find an  $\epsilon$-accurate solution in at most  $$\min\left\{\tilde{\mathcal{O}}(\sigma_2^{-5/6}\epsilon^{-4/3}), \tilde{\mathcal{O}}(\sigma_2^{-1/2}n)\right\}  \text{  {\rm equivalent stochastic gradient iterations.}}$$	
\end{corollary}

\section{The accelerated inexact proximal cubic regularized Newton method}\label{sec:acc-cubic}

\begin{algorithm} [H]
	\caption{Accelerated inexact proximal cubic regularized Newton method (AIPCNM)}
	\label{alg:AICNM}
	\begin{algorithmic}[1]

\STATE Input:
$x_0 \in \bbr^n$,  $\eta=4L_3,  C_1>0, C_2>0$, a sequence $\{A_t\}_{t\ge 0}$
\STATE  Set $v_0=x_0$
\STATE Set $ \psi_0(x) = \frac{C_1}{2}\|x-x_0\|^2+ \frac{C_2}{3}\|x-x_0\|^3 $
\FOR{$t=0,1,2,\ldots$}
\STATE Set $a_t = A_{t+1}-A_t$
\STATE Obtain inexact gradient $g_t$ and Hessian $H_t$ , where $H_t$ satisfies Assumption \ref{ass:H}
\STATE Set $y_t = (1 - \alpha_t) x_t + \alpha_t v_t$, where $\alpha_t=\frac{a_t}{A_t+a_t}$
\STATE Find an approximate solution $x_{t+1}$ of  $\min_{x\in\bbR^d}\tilde{f}_{\eta}(x;y_t)$, where $\tilde{f}_{\eta}(x;y)$ is defined in \eqref{eq:subprob}
\STATE Obtain $g_{t+1}^{\prime}$ that satisfies Assumption~\ref{ass:g-prime}
\STATE  Find  $v_{t+1}=\argmin_{x \in \bbr^d}\psi_{t+1}(x)$, where 
\begin{align}
\psi_{t+1}(x)=&\psi_t(x)+a_t\big(f(x_{t+1})+\langle g_{t+1}^{\prime}, x-x_{t+1} \rangle+h(x)\big).	 \label{eq:psi}
\end{align}

\ENDFOR
	\end{algorithmic}
\end{algorithm}

In Algorithm \ref{alg:AICNM}, we propose the AIPCNM method. To ensure convergence, we make Assumptions \ref{ass:H} and \ref{ass:g-prime}.

\begin{assumption}\label{ass:H}
Let $\{y_t\}_{t \ge 0}$ be generated by Algorithm~\ref{alg:AICNM}. Then we have
\begin{eqnarray}
\frac{\mu_t}{2}\preceq H_t - \nabla^2 f(y_t)\preceq \mu_t, 
\end{eqnarray}
where $\{\mu_t\}_{t\ge 0}$ is a positive sequence.
\end{assumption}

\begin{assumption}\label{ass:g-prime}
$g_{t+1}^{\prime}$ is an unbiased estimation of $\nabla f(x_{t+1})$, $i.e.,$ $\E[g_{t+1}^{\prime}] = \nabla f(x_{t+1})$. 

\end{assumption}

Then by extending \citep[Lemma 2.1]{nunes2018accelerated}, we obtain Lemma  \ref{eq:acc-key}, which is the key lemma to extend the conclusion of the exact APCNM to the inexact case.  
\begin{lemma}\label{eq:acc-key}
Let $\{x_t\},\{y_t\}$ be generated by Alg. \ref{alg:AICNM}. Denote $$q_{t+1} \overset{\rm def}{=} g_t-\nabla f(y_t)+\nabla F(x_{t+1}) - \nabla \tilde{f}_{\eta}(x_{t+1}; y_t),$$ then we have
\begin{eqnarray*}
q_{t+1}^T(y_t-x_{t+1})
&\ge&\min\left\{ \frac{\|q_{t+1}\|^2}{3\mu_t}  , \sqrt{\frac{\|q_{t+1}\|^3}{4L_3+2\eta}} \right\}+\frac{\mu_t}{4}\|y_t-x_{t+1}\|^2.
\end{eqnarray*}
\end{lemma}
In Lemma \ref{eq:acc-key}, we define $q_{t+1}$ as an proxy to the $\nabla F(x_{t+1}$ of the exact case. Then based on  Lemma \ref{eq:acc-key}, we prove Theorem \ref{thm:prime-result}.

\begin{theorem}\label{thm:prime-result}
Assume for $t\ge0$, $\|v_{t+1}-v_t\|\le R$. Meanwhile
assume the constants $C_1>0, C_2>0$ in Algorithm \ref{alg:AICNM} satisfy for $t\ge0$, 
\begin{eqnarray*}
C_1&\ge&\max_{t\ge0}\left\{\frac{9\mu_ta_t^2}{2A_{t+1}}-\frac{2}{3}A_t\sigma_2\right\}\\
C_2 &\ge&\max_{t\ge0}\left\{ \frac{32a_t^3L_3}{3A_{t+1}^2}-\frac{A_t\sigma_2}{R}\right\}.
\end{eqnarray*} 
Then if sequences $\{x_t\}$, $\{v_t\}$ are generated by Algorithm \ref{alg:AICNM}, then for all $t>0$, we have
\begin{eqnarray}
\E[F(x_t) - F(x^*)]\le \frac{1}{A_t}\left( \frac{C_1}{2}\|x_0-x^*\|^2+ \frac{C_2}{3}\|x_0-x^*\|^3\right) + \frac{1}{A_t} \sum_{i=1}^t G_i,
\end{eqnarray}
where 
\begin{eqnarray}
G_i &=& \left(\frac{A_{i+1}}{\mu_i}+\frac{9a_i^2}{2(3C_1+2A_i\sigma_2)}\right) \|g_i-\nabla f(y_i) - \nabla \tilde{f}_{\eta}(x_{i+1}; y_i)\|^2\nonumber\\
 &&+  \frac{9a_i^2}{2(3C_1+2A_i\sigma_2)}\| g_{i+1}^{\prime}- \nabla f(x_{i+1})\|^2,
\end{eqnarray}
and the expectation is taken on all the history of the randomness of $g_{i+1}^{\prime}$ from $i=0$ to $t-1$.
\end{theorem}
In Theorem \ref{thm:prime-result}, $\mu_i$ bound the error of $H_i$, $\|g_i-\nabla f(y_i) - \nabla \tilde{f}_{\eta}(x_{i+1}; y_i)\|$ bounds the error of the inexact gradient $g_i$ and the inexact solution\footnote{$\|\nabla \tilde{f}_{\eta}(x_{i+1}; y_i)\|$ can be used as a measure of the subsolver, which can be bounded by $\tilde{f}_{\eta}(x_{i+1}; y_i) - \tilde{f}_{\eta}(\tilde{x}; y_i)$ by the Lipschitz gradient property} $x_{i+1}$  and $\| g_{i+1}^{\prime}- \nabla f(x_{i+1})\|$ bounds the error of $g_{i+1}^{\prime}$.

\begin{theorem}[The convex case]\label{thm:acc-convex}
Assume that $\|x_0-x^*\|\le D$.
If we set 
\begin{equation}
\begin{cases}
&\forall 0\le i, A_{i}=\frac{i(i+1)(i+2)}{6},\mu_i = \frac{L_3D}{i+2} \\
&C_1 = 7L_3D, C_2 = 48L_3 \\
&\|g_i-\nabla f(y_i) - \nabla \tilde{f}_{\eta}(x_{i+1}; y_i)\|\le  \frac{L_3D^2}{\sqrt{2t}(i+2)^2}\\
&\| g_{i+1}^{\prime}- \nabla f(x_{i+1})\| \le \frac{2L_3D^2}{\sqrt{t}(i+2)^2},\\
\end{cases}\label{eq:cond11}
\end{equation}
 then if $F(x)$ is convex, we have
\begin{eqnarray*}
\E[F(x_t) - F(x^*)]&\le& \frac{129L_3D^3}{t(t+1)(t+2)}, \label{eq:cond12}
\end{eqnarray*}
where the expectation is taken on all the history of the randomness of $g_{i+1}^{\prime}$ from $i=0$ to $t-1$.
\end{theorem}
\begin{proof}
By using the setting in  \eqref{eq:cond11}  and Theorem \ref{thm:prime-result}, we can obtain \eqref{eq:cond12} directly. 
\end{proof}

\begin{theorem}[The strongly convex case]\label{thm:acc-strong-convex}
If $F(x)$ is $\sigma_2$-strongly convex, by setting 
\begin{eqnarray}
\rho=\min\left\{1,  \frac{3^{1/3}}{2}\left(\frac{\sigma_2}{L_3R}\right)^{1/3}\right\}, 
\end{eqnarray}
and set
\begin{equation}
\begin{cases}
&A_0 =0 \\
&\forall  i\ge 1, A_{i} = (1+\rho)^i, \mu_i =\mu_0= \frac{32(L_3R)^{2/3}\sigma_2^{1/3}}{27\cdot 3^{2/3}}\\
&C_1=\frac{9\mu_0(1+\rho)}{2}, C_2=\frac{32(1+\rho)L_3}{3} \\
&\|g_i-\nabla f(y_i) - \nabla \tilde{f}_{\eta}(x_{i+1}; y_i)\|\le  \left(\frac{1}{\mu_0}+\frac{9\rho^2}{4\sigma_2}\right)^{-1/2}(1+\rho)^{-i/2+1}t^{-1/2}\\
&\| g_{i+1}^{\prime}- \nabla f(x_{i+1})\| \le \frac{2\sigma_2^{1/2}}{3\rho} (1+\rho)^{-i/2+1}t^{-1/2}\\
\end{cases}\label{eq:cond21}
\end{equation}
then we have for $t\ge 1$, 
\begin{eqnarray}
\E[F(x_t) - F(x^*)]&\le& (1+\rho)^{-(t+1)}\left( \frac{9\mu_0}{4}\|x_0-x^*\|^2 + \frac{32L_3}{9}\|x_0-x^*\|^3+2\right),\label{eq:cond22}
\end{eqnarray}
where the expectation is taken on all the history of the randomness of $g_{i+1}^{\prime}$ from $i=0$ to $t-1$.
\end{theorem}
\begin{proof}
By using the setting in  \eqref{eq:cond21}  and Theorem \ref{thm:prime-result}, we can obtain \eqref{eq:cond22} directly. 
\end{proof}

Finally, we give Corollaries  \ref{thm:acc-online} and \ref{thm:acc-finite} to show the overall complexity in the online stochastic setting. 
\begin{corollary}[The online stochastic setting]\label{thm:acc-online}
Suppose that Assumptions \ref{ass:gradient-Hessian} and \ref{ass:g-H} hold. Then if $F(x)$ is convex, then with the probability $1-\delta$, AIPCNM can find an $\epsilon$-accurate solution in at most  $$\tO(\epsilon^{-2}) \text{  \rm{equivalent stochastic gradient iterations}}. $$

If $F(x)$ is $\sigma_2$-strongly convex, IPCNM can find an  $\epsilon$-accurate solution in at most  $$\tO(\sigma^{-2/3}\epsilon^{-1}) \text{  \rm{equivalent stochastic gradient iterations}}.$$  
\end{corollary}

\begin{corollary}[The finite-sum setting]\label{thm:acc-finite}
Suppose that Assumptions \ref{ass:gradient-Hessian} and \ref{ass:g-H} hold. Then if $F(x)$ is convex, then with the probability $1-\delta$, AIPCNM can find an $\epsilon$-accurate solution in at most  $$\min\{\tilde{\mathcal{O}}(\epsilon^{-2}), \tilde{\mathcal{O}}(n\epsilon^{-1/3})\}  \text{  \rm{equivalent stochastic gradient iterations}}.$$ 

If $F(x)$ is $\sigma_2$-strongly convex, IPCNM can find an  $\epsilon$-accurate solution in at most  $$\min\left\{\tilde{\mathcal{O}}(\sigma_2^{-2/3}\epsilon^{-1}),\tilde{\mathcal{O}}(\sigma_2^{-1/3}n)\right\}  \text{  \rm{equivalent stochastic gradient iterations}}.$$   	
\end{corollary}


\section{The Proximal SVRG with Cubic Regularization}\label{sec:cubic-svrg}
In this section, we propose an efficient algorithm called Cubic Proximal Stochastic Variance Reduced Gradient method (Cubic-Prox-SVRG) in Algorithm \ref{alg:cubic} to solve the cubic regularized second-order subproblem $\min_{x\in\bbR^d}\tilde{f}_{\eta}(x, y)$, where $\tilde{f}_{\eta}(x, y)$ is defined in \eqref{eq:subprob}. In this section we assume that the inexact Hessian $H$ is obtained by subsampling by Assumption \ref{ass:g-H}, $i.e,$ $H\overset{\defi}{=}\frac{1}{n}\sum_{i=1}^n H_i$, where $n$ is the number of subsamples, and we assume for any $v\in\bbR^d$, the cost of $H_i v$ is $O(d)$. 

Then we reformulate the subproblem $\min_{x\in\bbR^d}\tilde{f}_{\eta}(x, y)$ as
\begin{eqnarray}
\min_{w\in\bbR^d}P(w)  \overset{\defi}{=} \frac{1}{n}\sum_{i=1}^n \psi_i(w) +  r(w),
\end{eqnarray}
where $w \overset{\defi}{=} x-y, \psi_i(w) \overset{\defi}{=} w^TH_i w, r(w)\overset{\defi}{=} \frac{\eta}{3}\|\w\|^3+h(w+y)$.

The Cubic-Prox-SVRG algorithm is motivated by the uniform property of degree $3$ 
\begin{eqnarray}
\frac{1}{3}\|w\|^3\ge \frac{1}{3}\|u\|^3+\langle \nabla \frac{1}{3}\|u\|^3, w-u\rangle + \frac{1}{6}\|w-u\|^3
\end{eqnarray}
of the cubic regularizer $\frac{1}{3}\|w\|^3$ \cite{nesterov2008accelerating}.  

 Assume $H\succeq 0$,  $h(w+y)$ is $\sigma_2$-strongly convex $ (\sigma_2\ge0)$. Then $P(w)$ is $\sigma_2$-strongly convex and $\frac{\eta}{2}$-uniformly convex of degree $3$. Meanwhile denote $ w^*\overset{\defi}{=}\argmin_{w\in\bbR^d}P(w).$


For two points $w\in\bbR^d, w^{\prime}\in\bbR^d$, the uniform convexity of degree $3$ is equivalent to $\frac{1}{2}\|w-w^{\prime}\|$-strong convexity. Therefore, the $3$-rd order uniform convexity is 
stronger when the two points are far away from each other and  is weaker than the strong convexity when the two points are close each other. Meanwhile, it is known that if $P(w)$ is smooth and strongly convex, gradient descent methods can converge with a linear rate \cite{nesterov1998introductory}. Combing the two facts, when $P(w)$ is smooth and $3$-rd order uniformly convex, we may obtain an gradient based algorithm with a two-stage convergence rate: a superlinear rate when the iterative solution is far away from the optimal point and a sublinear rate when they are close each other.

To verify this intuition, we propose a new algorithm called Cubic regularized Proximal SVRG (Cubic-Prox-SVRG) in Alg. \ref{alg:cubic}, which is a variant of the well-known Prox-SVRG algorithm \cite{xiao2014proximal}.  Compared with Prox-SVRG in \cite{xiao2014proximal}, the difference is only the number of the inner iteration $M_s$ and the learning rate for each outer iteration $\tau_s$. In Theorem \ref{thm:converge}, we give the two-stage convergence rate of Cubic-Prox-SVRG.

\begin{algorithm}[!ht]
    \caption{Cubic proximal stochastic variance reduced gradient}
\begin{algorithmic}[1]
	\STATE Initialization: $\tilde{w}_0 = 0, m = O(n), \tau_0 = 0.1/L_2$
	\STATE $Q=\{q_1, q_2,\ldots, q_k, \dots, q_n\},$ where $q_k\overset{\defi}{=}\frac{\|H_k\|}{\sum_{i=1}^n \|H_i\|}$;
	 $L_2 \overset{\defi}{=} \frac{1}{n}\sum_{i=1}^n \|H_{i}\|$
	 
	 $\kappa_2 \overset{\defi}{=} \frac{L_2}{\sigma_2}, \kappa_3\overset{\defi}{=}\frac{L_2}{2}\left(\frac{12}{\eta}\right)^{2/3}$
	 
	 $M_s\overset{\defi}{=} \lceil 100\min\{\kappa_2,\kappa_3\max\{m, (P(\tilde{w}_{s-1}) - P(w^*))^{-1/3}\}\}\rceil;$

	$ \tau_s \overset{\defi}{=} \tau_0\min\{1,m^{-\frac{1}{2}}(P(\tilde{w}_{s-1}) - P(w^*))^{-1/6}\}$
    \vspace{0.02in}
    \FOR{$s =1, 2, 3, \ldots$}


	

	\STATE $\tilde{w} = \tilde{w}_{s-1}$
    \STATE $\tilde{\mu}=\frac{1}{n}\sum_{i=1}^n \nabla \psi_i(\tilde{w})$
   	\STATE $w_0 = \tilde{w}$
	\FOR{$k=1, 2, 3,..., M_s$}
	\STATE Pick $i_k\in \{1, ..., n\}$ randomly according to $Q$
	\STATE $\tilde{\nabla}_k =( \nabla \psi_{i_k}(w_{k-1}) -\nabla \psi_{i_k}(\tilde{w}))/(q_{i_k}n)+ \tilde{\mu}$
	\STATE ${w}_k =\text{prox}_{r}(w_{k-1} -\tau_s \tilde{\nabla}_k   ) $
	\ENDFOR
	\STATE $\tilde{w}_s = \frac{1}{M_s}\sum_{k=1}^{M_s} w_k$
	
    \ENDFOR
\end{algorithmic}\label{alg:cubic}
\end{algorithm}

\begin{theorem}\label{thm:converge}
Assume that $s_1$ is the smallest number of outer iteration that satisfies $$P(\tilde{w}_{s_1})-P(w^*)\le \frac{1}{m^3}.$$ 
Then it follows that 
\begin{eqnarray*}
&&\!\!\!\!\!\!\!\!\!\E[P(\tilde{w}_s) - P(w^*)]\le\begin{cases}
\left(\frac{\rho}{\sqrt{m}}\right)^{6\left(1-\left(\frac{5}{6}\right)^s\right)}\left(P(\tilde{w}_0)-P(w^*)\right)^{\left(\frac{5}{6}\right)^s}, &\text{ {\rm if } } s\le s_1 \\
\rho^{s-s_1} \left(P(\tilde{w}_{s_1})-P(w^*)\right), & \text{ {\rm if } } s> s_1 \\
\end{cases}
\end{eqnarray*}
where $\rho \overset{\rm def}{=} \frac{1}{100L_2\tau_0(1-4L_2\tau_0)} +\frac{4L_2\tau_0(100\kappa_2 m+1)}{100(1-4L_2\tau_0)\kappa_2 m}<1$.
\end{theorem}
It should be noted that $\eta=\O(1), L_2=\O(1)$. Then by the definition, $\kappa_3$ is an $O(1)$ constant.  

By Theorem \ref{thm:converge}, the outer iteration $\tilde{w}_s$ will converge to a neighborhood of $w^*$ in a superlinear rate until $P(\tilde{w}_{s_1})-P(w^*)\le \frac{1}{m^3}$. Then it will converge in a linear rate. In the convex setting,  the number of stochastic samples in IPCNM and AIPCNM will be $\tO(t^2)$ in the $t$-th iteration. To ensure the convergence rate in Theorems \ref{thm:nonstrong} and  \ref{thm:acc-convex}, we need the solving accuracies of the subproblem are $\O(\frac{1}{t^3)}$ and $\O(\frac{1}{t^5)}$ respectively, while by Theorem \ref{thm:converge}, if $n=\O(t^2)$, then Cubic-Prox-SVRG can converge to an $\O(1/t^6)$-accurate solution in a superlinear rate.  In the strongly convex setting, Cubic-Prox-SVRG will finally converge in a linear rate and thus satisfies Assumption \ref{ass:appro}. 

\bibliographystyle{alpha}
\bibliography{references}

\clearpage
\onecolumn
\appendix

\section{Some technical results}
\begin{lemma}[\cite{nesterov2006cubic}]\label{lem: Lipschitz}
Suppose $f(x)$ has $L_3$-Lipschitz Hessians. Then we have $\forall x\in\bbR^d, y\in \bbR^d$, 
\begin{eqnarray}
f(y)&\le& f(x)+\langle \nabla f(x), y-x\rangle+ \frac{1}{2}\langle \nabla^2 f(x)(y-x), y-x\rangle + \frac{L_3}{6}\|y-x\|^3.
\end{eqnarray}
and 
\begin{eqnarray}
f(y)&\ge& f(x)+\langle \nabla f(x), y-x\rangle+ \frac{1}{2}\langle \nabla^2 f(x)(y-x), y-x\rangle - \frac{L_3}{6}\|y-x\|^3.
\end{eqnarray}
\end{lemma}
\begin{lemma}[\cite{nesterov2008accelerating}]\label{lem:inner-prod}
For any $x\in\bbR^d$ and $y\in\bbR^d$, we have
\begin{eqnarray}
|\langle x, y\rangle|\le \frac{1}{p}\sigma\|x\|^p + \frac{p-1}{p}\left(\frac{1}{\sigma}\right)^{\frac{1}{p-1}}\|y\|^{\frac{p}{p-1}},
\end{eqnarray}
where $\sigma>0, p \ge 2$.
\end{lemma}

\begin{lemma}[Vector Bernstein Inequality, Lemma 18 in \cite{kohler2017sub}]\label{lem:vector}Let $x_1,\ldots, x_n$ be independent vector-valued random variables with common dimension d and assume that each one is centered, uniformly bounded and also the variance is bounded above:
\begin{equation}
\E[x_i] = 0, \quad\quad \|x_i\|_2\le \gamma_1,   \quad\quad\E[\|x_i\|^2]\le \tau_1^2.
\end{equation}
Let 
\[
z = \frac{1}{n}\sum_{i=1}^n x_i,
\]
then we have for $0<\epsilon<\tau_1^2/\gamma_1$,
\[
P(\|z\|\ge \epsilon)\le \exp\left(-n\frac{\epsilon^2}{8\tau_1^2}+\frac{1}{4}\right).
\]
\end{lemma}

\begin{lemma}[Matrix Bernstein Inequality, Lemma 19 in \cite{kohler2017sub}]\label{lem:matrix}Let $A_1 , \ldots, A_n$ be independent random Hermitian matrices with common dimension d × d and assume that each one is centered, uniformly bounded and also the variance is bounded above:
\begin{equation}
\E[A_i] = 0, \quad\quad \|A_i\|_2\le \gamma_2,   \quad\quad\|\E[A_i^2]\|\le \tau_2^2.
\end{equation}
Let 
\[
Z = \frac{1}{n}\sum_{i=1}^n A_i,
\]
then we have for $0<\epsilon<2\tau_2^2/\gamma_2$,
\[
P(\|z\|\ge \epsilon)\le2d \exp\left(-n \frac{\epsilon^2}{4\tau_2^2} +\frac{1}{4}\right).
\]

\end{lemma}

\section{Proof of the theorems in Section \ref{sec:cubic}}
\begin{proof}[Proof of Lemma \ref{lem:basic-f}]
First, it follows that 
\begin{eqnarray}
F(x_{t+1})&\overset{\circlenum{1}}{\le}& \tilde{f}_{\eta}(x_{t+1}; x_t) + \frac{L_3-\eta}{6}\|x_{t+1}-x_t\|^3 + \frac{1}{2}\langle(\nabla^2 f(x_t) - H_t)(x_{t+1}-x_t), x_{t+1}-x_t\rangle \nonumber\\
&&+ \langle \nabla f(x_t) - g_t, x_{t+1}-x_t\rangle \nonumber\\
&\overset{\circlenum{2}}{\le}&  \tilde{f}_{\eta}(x_{t+1}; x_t) + \frac{L_3-\eta}{6}\|x_{t+1}-x_t\|^3 + \frac{1}{2}\|\nabla^2 f(x_t) - H_t\|\|x_{t+1}-x_t\|^2 \nonumber\\
&&  
 + \langle \nabla f(x_t) - g_t, x_{t+1}-x_t\rangle \nonumber\\
&\overset{\circlenum{3}}{\le}&  \tilde{f}_{\eta}(x_{t+1}; x_t) + \frac{L_3-\eta}{6}\|x_{t+1}-x_t\|^3+     \frac{2}{3L_3^2}\|\nabla^2 f(x_t) - H_t\|^3 \nonumber\\
&&+ \frac{L_3}{6}\|x_{t+1}-x_t\|^3+ \frac{L_3}{6}\|x_{t+1}-x_t\|^{3}  + \frac{2}{3}\left(\frac{2}{L_3}\right)^{\frac{1}{2}}\|\nabla f(x_t) - g_t \|^{\frac{3}{2}}, \nonumber\\
&\le&\tilde{f}_{\eta}(x_{t+1}; x_t) + \frac{3L_3-\eta}{6}\|x_{t+1}-x_t\|^3+\frac{2}{3L_3^2}\|\nabla^2 f(x_t) - H_t\|^3 \nonumber\\
&&+ \frac{2}{3}\left(\frac{2}{L_3}\right)^{\frac{1}{2}}\|\nabla f(x_t) - g_t \|^{\frac{3}{2}}      \nonumber
 \label{eq:1}
\end{eqnarray} 
where \circlenum{1} is by Lemma \ref{lem: Lipschitz}, \circlenum{2} is by the Cauchy inequality and the definition of the spectral norm $\|\nabla^2 f(x_t)-H_t\|$
, \circlenum{3} is by using Lemma \ref{lem:inner-prod} twice.  

Second, it follows that 
\begin{eqnarray}
\tilde{f}_{\eta}(x_{t+1}; x_t) &{=}& \tilde{f}_{\eta}(x_{t+1}^*; x_t) + (\tilde{f}_{\eta}(x_{t+1}; x_t)-\tilde{f}_{\eta}(x_{t+1}^*; x_t)) \nonumber\\
&\overset{\circlenum{1}}{\le}&  \tilde{f}_{\eta}(x; x_t) + (\tilde{f}_{\eta}(x_{t+1}; x_t)-\tilde{f}_{\eta}(x_{t+1}^*; x_t)) , \label{eq:2}
\end{eqnarray}
where  \circlenum{1} is by the optimality condition of $x_{t+1}^* = \argmin_{x\in\bbR^d}\tilde{f}_{\eta}(x; x_t)$.  
 
Third, we have
\begin{eqnarray}
\tilde{f}_{\eta}(x; x_t) &\overset{\circlenum{1}}{\le}& F(x) + \frac{L_3+\eta}{6}\|x-x_t\|^3   -\frac{1}{2}\langle (\nabla^2 f(x_t)-H_t)(x-x_t), x-x_t\rangle \nonumber\\
&&- \langle \nabla f(x_t)- g_t, x-x_t\rangle \label{eq:3} \nonumber\\
&\overset{\circlenum{2}}{\le}& F(x) + \frac{L_3+\eta}{6}\|x-x_t\|^3   +\frac{1}{2}\|\nabla^2 f(x_t)-H_t\|\| x-x_t\|^2 - \langle \nabla f(x_t)- g_t, x-x_t\rangle \label{eq:3} \nonumber\\
 &\overset{\circlenum{3}}{\le}&  F(x) + \frac{L_3+\eta}{6}\|x-x_t\|^3+\frac{2}{3L_3^2}\|\nabla^2 f(x_t) - H_t\|^3 + \frac{L_3}{6}\|x-x_t\|^3 \nonumber\\
&&+ \frac{L_3}{6}\|x-x_t\|^{3}  + \frac{2}{3}\left(\frac{2}{L_3}\right)^{\frac{1}{2}}\|\nabla f(x_t) - g_t \|^{\frac{3}{2}}, \nonumber\\
 &\overset{\circlenum{4}}{\le}&  F(x) + \frac{3L_3+\eta}{6}\|x-x_t\|^3 + \frac{2}{3L_3^2}\|\nabla^2 f(x_t) - H_t\|^3 +  \frac{2}{3}\left(\frac{2}{L_3}\right)^{\frac{1}{2}}\|\nabla f(x_t) - g_t \|^{\frac{3}{2}}  \nonumber\\
 \end{eqnarray} 
where \circlenum{1} is by Lemma \ref{lem: Lipschitz}, \circlenum{2} is by the Cauchy inequality and the definition of the spectral norm $\|\nabla^2 f(x_t)-H_t\|$
, \circlenum{3} is by using Lemma \ref{lem:inner-prod} twice.  

 Fourth, by \eqref{eq:1}-\eqref{eq:3} and setting $\eta= 3L_3$, it follows that
 \begin{eqnarray}
 F(x_{t+1}) &\overset{\circlenum{1}}{\le}& F(x)+L_3\|x-x_t\|^3+  \frac{4}{3L_3^2}\|\nabla^2 f(x_t) - H_t\|^3 +  \frac{4}{3}\left(\frac{2}{L_3}\right)^{\frac{1}{2}}\|\nabla f(x_t) - g_t \|^{\frac{3}{2}}\nonumber\\
 &&+(\tilde{f}_{\eta}(x_{t+1}; x_t)-\tilde{f}_{\eta}(x_{t+1}^*; x_t))\nonumber\\
 &\overset{\circlenum{2}}{=}&  F(x)+L_3\|x-x_t\|^3+  E_t,\nonumber
 \end{eqnarray}
where \circlenum{1} is by summing  \eqref{eq:1}-\eqref{eq:3} and setting $\eta= 3L_3$, \circlenum{2} is by the definition of $$E_t\overset{\defi}{=} \frac{4}{3L_3^2}\|\nabla^2 f(x_t) - H_t\|^3 +  \frac{4}{3}\left(\frac{2}{L_3}\right)^{\frac{1}{2}}\|\nabla f(x_t) - g_t \|^{\frac{3}{2}}
 +(\tilde{f}_{\eta}(x_{t+1}; x_t)-\tilde{f}_{\eta}(x_{t+1}^*; x_t)).$$

Lemma \ref{lem:basic-f} is proved.

\end{proof}

\begin{proof}[Proof of Theorem \ref{thm:nonstrong}]
First, let $x\overset{\defi}{=} x_t+\alpha_t(x^*-x_t) (0\le \alpha_t\le 1)$. Then by the convexity of $F(x)$, we have
\begin{eqnarray}
F(x_{t+1})&\le&(1-\alpha_t)F(x_t) + \alpha_t F(x^*)+ L_3 \alpha_t^3\|x^*-x_t\|^3 + E_t\nonumber
\end{eqnarray}
Second, by setting $\alpha_t = \frac{3}{t+3} (t\ge 0)$ and using Assumption \ref{ass:x} that $\|x_t-x^*\|\le D$, then we have 
\begin{eqnarray}
F(x_{t+1}) - F(x^*)\le (1-\alpha_t)(F(x_t) - F(x^*))+\frac{27L_3D^3}{(t+3)^2} + E_t.\label{eq:non-1}
\end{eqnarray}
Third, set 
$$A_{t} \overset{\defi}{=} 
\begin{cases}
1, & t = 0 \\
\Pi_{i=1}^t(1-\alpha_i), & t\ge 1
\end{cases}$$  and $\forall t\ge 1$, dividing $A_{t}$ on both sides of \eqref{eq:non-1}, it follows that
\begin{eqnarray}
F(x_1) - F(x^*)&\le& L_3D^3+E_0, \nonumber\\
\frac{F(x_{t+1})-F(x^*)}{A_{t}}&\le&\frac{F(x_{t})-F(x^*)}{A_{t-1}}+\frac{27L_3D^3	}{(t+3)^3A_{t}}+\frac{E_t}{A_t}.\label{eq:non-2}
\end{eqnarray}
By the definition of $A_{t}$, $\forall t\ge 0$, we have $A_{t} = \frac{6}{(t+3)(t+2)(t+1)}$. 
Summing up both sides of \eqref{eq:non-2} from $i=1$ to $t$, we have
\begin{eqnarray}
\frac{F(x_{t+1})-F(x^*)}{A_{t}}&\le& \frac{F(x_{1})-F(x^*)}{A_{0}} + \sum_{i=1}^{t}  \frac{27L_3D^3	}{(i+3)^3A_{i}} +  \sum_{i=1}^{t} \frac{E_i}{A_i}\nonumber\\
&\le& L_3D^3+E_0+ \frac{9}{2}L_3D^3\sum_{i=1}^t \frac{(i+2)(i+1)}{(i+3)^2} +  \sum_{i=1}^{t} \frac{E_i}{A_i}\nonumber\\
&\le& L_3D^3+  \frac{9}{2}L_3D^3 \sum_{i=1}^t \left(1 - \frac{1}{i+2}\right) +  \sum_{i=0}^{t} \frac{E_i}{A_i}\nonumber\\
&\le&  \frac{9}{2}tL_3D^3 +  \sum_{i=0}^{t} \frac{E_i}{A_i}.\nonumber
\end{eqnarray}
Therefore we have 
\begin{eqnarray}
F(x_{t+1}) - F(x^*)\le\frac{27L_3D^3}{(t+3)(t+2)} + \frac{1}{(t+1)(t+2)(t+3)}  \sum_{i=0}^{t} ((i+1)(i+2)(i+3))E_i.\nonumber
\end{eqnarray}
Theorem \ref{thm:nonstrong} is proved.
\end{proof}

\vgap

\begin{proof}[Proof of Theorem \ref{thm:strong}]
First, let $x\overset{\defi}{=} x_t+\alpha_t(x^*-x_t) (0\le \alpha_t\le 1)$. Then by the strong convexity of $F(x)$, we have
\begin{eqnarray}
F(x_{t+1})&\le&(1-\alpha_t)F(x_t) + \alpha F(x^*) -\frac{\sigma_2\alpha_t(1-\alpha_t)}{2}\|x^*-x_t\|^2  + L_3 \alpha_t^3\|x^*-x_t\|^3 + E_t. \nonumber
\end{eqnarray}
Then  we have 
\begin{eqnarray}
F(x_{t+1}) - F(x^*)&\le& (1-\alpha_t)(F(x_t)-F(x^*))- \frac{\alpha_t}{2}\|x_t-x^*\|^2(\sigma_2(1-\alpha_t) - 2L_3\alpha_t^2\|x_t-x^*\|) + E_t\nonumber\\
&\overset{\circlenum{1}}{\le}& (1-\alpha_t)(F(x_t)-F(x^*))- \frac{\alpha_t}{2}\|x_t-x^*\|^2(\sigma_2(1-\alpha_t) - 2L_3\alpha_t^2D) + E_t,\nonumber
\end{eqnarray}
where \circlenum{1} is by the assumption that $\forall t\ge 0$, $\|x_t-x^*\|\le D$.

Then $\forall t\ge0$, let $\alpha_t \overset{\rm def}{=} \min\left\{\frac{1}{3}, \sqrt{\frac{\sigma_2}{3L_3D}}\right\}$. Then we have 
\begin{eqnarray}
\sigma(1-\alpha_t) - 2L_3\alpha_t^2\|x_t-x^*\|\ge 0.\nonumber
\end{eqnarray}
Thus, one has
\begin{eqnarray}
F(x_{t+1}) - F(x^*)\le (1-\alpha_t)(F(x_t)-F(x^*)) + E_t.\nonumber
\end{eqnarray}
For $t_0\ge t\ge -1$, let
\begin{equation}
A_t \overset{\rm def}{=} 
\begin{cases}
1, & t = -1\\
\Pi_{i=0}^t (1-\alpha_i), & t\ge 0.	
\end{cases}	\nonumber
\end{equation}

 It follows that for $ t_0-1\ge t\ge 0$,
\begin{eqnarray}
\frac{F(x_{t+1}) - F(x^*)}{A_t}\le \frac{F(x_{t}) - F(x^*)}{A_{t-1}} + \frac{E_t}{A_t}. \nonumber
\end{eqnarray}
Then we have 
\begin{eqnarray}
F(x_{t+1}) - F(x^*)\le A_{t}\left(F(x_{0}) - F(x^*)\right) + A_t \sum_{i=0}^{t}\frac{E_i}{A_i}.\nonumber
\end{eqnarray}
Then by the definition of $A_t$, Theorem \ref{thm:strong} is proved.
\end{proof}

\begin{proof}[Proof of Theorem \ref{thm:superlinear}]

In Lemma \ref{lem:basic-f}, by setting $\alpha_{t-1}=1$ and using the $\sigma_2$-strongly convex property of $F(x)$, then we have 
\begin{eqnarray}
F(x_{t})-F(x^*)&\le& L_3\|x_{t-1}-x^*\|^3+E_{t-1}\nonumber\\ 
 &\le& L_3\left(\frac{2}{\sigma_2}(F(x_{t-1}) - F(x^*))\right)^{3/2} + E_{t-1}
\label{eq:strong-2}
\end{eqnarray}

By the definition of $\omega$, we can rearrange \eqref{eq:strong-2} to
\begin{eqnarray}
\frac{F(x_{t}) - F(x^*) }{\omega}\le \frac{1}{2}\left(\frac{F(x_{t-1}) - F(x^*) }{\omega}\right)^{3/2} + \frac{E_{t-1}}{\omega}.\label{eq:strong-3}
\end{eqnarray}

Now we use the mathematical induction method to show that under the assumption $E_t\le\frac{\omega}{2}\left(2/3\right)^{(3/2)^{t-t_0+1}}$, for $t\ge t_0$, 
\begin{equation}
F(x_{t}) - F(x^*)\le \omega\left(2/3\right)^{(3/2)^{t-t_0}}.\label{eq:strong10}
\end{equation}

First, by the definition of $t_0$, we have $F(x_{t}) - F(x^*)\le \frac{2}{3}\omega$. Therefore \eqref{eq:strong10} is true trivially for $t=t_0$. 

Second, assume that for some $t>t_0$ such that \eqref{eq:strong10} is true. Then by the assumption $E_t\le\frac{\omega}{2}\left(2/3\right)^{(3/2)^{t-t_0+1}}$, we have 
\begin{eqnarray}
\frac{F(x_{t+1})-F(x^*)}{\omega}&\overset{\circlenum{1}}{\le}&  \frac{1}{2}\left(\frac{F(x_{t}) - F(x^*) }{\omega}\right)^{3/2} + \frac{E_t}{\omega}\nonumber\\
&\overset{\circlenum{2}}{\le}&(2/3)^{(3/2)^{t-t_0+1}},\nonumber
\end{eqnarray}
where \circlenum{1} is by \eqref{eq:strong-3}, \circlenum{2} is by the induction assumption and the assumption of $E_t$.

Theorem \ref{thm:superlinear} is proved.

\end{proof}

\begin{proof}[Proof of Corollary \ref{thm:online}]
\begin{itemize}
\item If $F(x)$ is only convex,  
in Theorem \ref{thm:nonstrong}, for $1\le i\le t$, if we set $E_i \le \frac{27L_3D^3}{i(i+1)(i+2)}$, then we have
\begin{equation}
F(x_t) - F(x^*)\le \frac{54L_3D^3}{(t+1)(t+2)}. \label{eq:online1}
\end{equation}
Then by the definition of $E_i$ in Lemma \ref{lem:basic-f}, we can verify that $1\le i\le t$, the following setting 
\begin{eqnarray}
\|\nabla^2 f(x_i) - H_i\|&\le& \frac{3L_3 D}{2(i+2)},\label{eq:online2}\\
\|\nabla f(x_i) - g_i\| &\le& \frac{9L_3 D^2}{8(i+2)^2}, \label{eq:online3}\\
\tilde{f}_{\eta_t}(x_{i+1}; x_i) -\tilde{f}_{\eta_t}(x_{i+1}^*; x_i)&\le& \frac{81L_3D^3}{4(i+2)^3},\label{eq:online4}
\end{eqnarray}
can make $E_i\le \frac{27L_3D^3}{i(i+1)(i+2)}$.

In the online stochastic setting, in order to let \eqref{eq:online1} hold with probability $1-\delta$,  the per-iteration failure probability should be set as
\begin{equation}
1-(1-\delta)^{\frac{1}{t}}\in \O\left(\frac{\delta}{t}\right). 
\end{equation}

By Assumptions \ref{ass:gradient-Hessian}, \ref{ass:g-H}, and Lemma \ref{lem:vector}, in order to let \ref{eq:online2} hold with probability $1-\frac{\delta}{t}$, it should be 
\begin{equation}
P\left(\|\nabla f(x_i) - g_i\|\ge\frac{9L_3 D^2}{8(i+2)^2}\right)\le \exp\left(-\hat{n}_{i1}\frac{1}{8\tau_1^2}\left(\frac{9L_3 D^2}{8(i+2)^2}\right)^2+\frac{1}{4}\right)\le\frac{\delta}{t}.\label{eq:online5}
\end{equation}
We can verify that 
\begin{equation}
\hat{n}_{i1}\ge \frac{(2+8\log \frac{t}{\delta})\tau_1^2 (i+2)^4}{L_3^2 D^4}
\end{equation}
can let \ref{eq:online5} hold. 
Then the total cost of stochastic gradient evaluations is $$\sum_{i=1}^t\hat{n}_{i1} = \O\left(
\frac{(2+8\log \frac{t}{\delta})\tau_1^2 (t+2)^5}{L_3^2 D^4}
\right).$$

Similarly, by Assumptions \ref{ass:gradient-Hessian}, \ref{ass:g-H}, and Lemma \ref{lem:matrix}, in order to let \ref{eq:online1} hold with probability $1-\frac{\delta}{t}$, it should be 
\begin{equation}
P\left(\|\nabla^2 f(x_i) - H_i\|\ge  \frac{3L_3 D}{2(i+2)}\right)\le2d\cdot \exp\left(-\hat{n}_{i2} \frac{1}{4\tau_2^2}\left( \frac{3L_3 D}{2(i+2)}\right)^2 +\frac{1}{4}\right)\le \frac{\delta}{t}.\label{eq:online6}
\end{equation}

We can verify that 
\begin{equation}
\hat{n}_{i2}\ge \frac{4(1+4\log \frac{2dt}{\delta})\tau_2^2 (i+2)^2}{9L_3^2 D^2}. 
\end{equation}
can let \ref{eq:online6} hold. 
Then the total cost of stochastic gradient evaluations is $$\sum_{i=1}^t\hat{n}_{i1} = \O\left(
\frac{4(1+4\log \frac{2dt}{\delta})\tau_2^2 (t+2)^3}{9L_3^2 D^2}
\right).$$

Because $\forall v\in\bbR^d$, $H_i v$ need $\hat{n}_{i2}$ stochastic Hessian-vector products, by Assumption \ref{ass:appro}, we need at most $\O(\hat{n}_{i2}\log (t+2))$  stochastic Hessian-vector products to find a solution $x_{t+1}$ such that \eqref{eq:online4} holds.

Define $\epsilon = \frac{54L_3D^3}{(t+1)(t+2)}$, then in order to let $F(x_t)-F(x^*)\le\epsilon$ with probability $1-\delta$, we need 
\begin{eqnarray*}
&&\O\left(
{\left(1+\log \frac{L_3^{1}D^{3}}{\delta\epsilon}\right)\tau_1^2 L_3^{1/2}D^{7/2} \epsilon^{-5/2}}
\right) \text{ stochastic gradient evaluations},\\
&&\O\left(
\left(1+\log \frac{dL_3D^3}{\delta\epsilon}\right)\tau_2^2 L_3^{-1/2}D^{5/2}\epsilon^{-3/2}
\right) \text{ stochastic Hessian-vector products},\\
&&\O\left(
\left(1+\log \frac{dL_3D^3}{\delta\epsilon}\right)\tau_2^2 L_3^{-1/2}D^{5/2}\epsilon^{-3/2}\log\frac{1}{\epsilon}
\right)\text{ stochastic Hessian-vector products  }
\end{eqnarray*}
of calling subsolver.

Therefore the dominant operation is stochastic gradient evaluation. By Assumption \ref{ass:g-H}, the cost of stochastic Hessian-vector product is not higher than stochastic gradient evaluation.
 Therefore, we need   $\tO(\epsilon^{-5/2})$ equivalent stochastic gradient evaluations to obtain an $\epsilon$-accurate solution.

\item If $F(x)$ is $\sigma_2$-strongly convex, in Theorem \ref{thm:strong}, for $1\le i\le t$, we use the following setting, we need 
\begin{equation}
 E_i \le \frac{(F(x_0)-F(x^*))(1-\alpha)^i}{t}, \\ 
\end{equation}
then by the definition of $E_i$ and using a similar analysis to the analysis in the convex setting, we need  $\tO(\sigma^{-5/6}\epsilon^{-4/3})$ equivalent stochastic gradient evaluations to obtain an $\epsilon$-accurate solution.

\end{itemize}


\end{proof}

\begin{proof}[Proof of Corollaries \ref{thm:finite}]
In the finite-sum setting, the analysis is similar to the online stochastic setting. The main difference is that the samples of gradient and Hessian can be larger than the number $n$ of the terms in the sum. Therefore, after the number of the samples evaluated by Lemmas \ref{lem:vector}  and \ref{lem:matrix} is larger than $n$, then we use the exact gradient in the following iteration. By this setting, we can know that in the convex setting, we only need $\min\{\tilde{\mathcal{O}}(\epsilon^{-5/2}),\tilde{\mathcal{O}}({n\epsilon^{-1/2}})\} $ in the convex setting and $\min\left\{\tilde{\mathcal{O}}(\sigma_2^{-5/6}\epsilon^{-4/3}), \tilde{\mathcal{O}}(\sigma_2^{-1/2}n)\right\} $ in the strongly convex setting.
\end{proof}

\section{Proof of the theorems in Section \ref{sec:acc-cubic}}

\begin{lemma}\label{lem:111}
For all $t\ge0$ and $x\in\bbR^d$, we have
\begin{eqnarray}
\E[\psi_t(x)]\le A_t F(x) + \frac{C_1}{2}\|x-x_0\|^2+\frac{C_2}{3}\|x-x_0\|^3,\label{eq:acc1}
\end{eqnarray}
where the expectation is taken on all the randomness from the $0$-th to $t$-th iteration.
\end{lemma}
\begin{proof}
Since $A_0=0$, $\forall x\in\bbR^d$, we have
\begin{eqnarray}
 \psi_0(x) = \frac{C_1}{2}\|x-x_0\|^2+ \frac{C_2}{3}\|x-x_0\|^3. 
\end{eqnarray}
Thus, \eqref{eq:acc1} is true for $t=0$. Then assume that \eqref{eq:acc1} is true for some $t\ge 0$. Then, fix the randomness before the randomness in obtaining $g_{t+1}^{\prime}$, we have 
\begin{eqnarray*}
\E[\psi_{t+1}(x)]&\overset{\circlenum{1}}{=}&\E[\psi_t(x)+a_t\big(f(x_{t+1})+\langle g_{t+1}^{\prime}, x-x_{t+1} \rangle+h(x)\big)\\
&\overset{\circlenum{2}}{=}&\psi_t(x)+a_t\big(f(x_{t+1})+\langle \nabla f(x_{t+1}), x-x_{t+1} \rangle+h(x)\big)\\
&\overset{\circlenum{3}}{\le}&\psi_t(x)+a_t\big(f(x)+h(x)\big),\\
&\overset{\circlenum{4}}{=}&\psi_t(x)+a_t F(x),
\end{eqnarray*}
where \circlenum{1} is by \eqref{eq:acc1}, \circlenum{2} is by the unbiasness of $g^{\prime}_{t+1}$ that $\E[g^{\prime}_{t+1}] =\nabla f(x_{t+1})$, \circlenum{3} is by the convexity of $f(x)$,  \circlenum{3} is by $F(x)=f(x)+h(x)$.

Then taking expectation on all the history, we have
\begin{eqnarray*}
\E[\psi_{t+1}(x)]&\le&\E[\psi_t(x)+a_t F(x)]\\
&\overset{\circlenum{1}}{\le}&A_t F(x) +  \frac{C_1}{2}\|x-x_0\|^2+\frac{C_2}{3}\|x-x_0\|^3+a_tF(x)\\
&\overset{\circlenum{2}}{\le}&A_{t+1} F(x) +  \frac{C_1}{2}\|x-x_0\|^2+\frac{C_2}{3}\|x-x_0\|^3,
\end{eqnarray*}
where \circlenum{1} is by the induction assumption, \circlenum{2} is by $A_{t+1}=A_t+a_t$.

\end{proof}

\begin{proof}[Proof of Lemma \ref{eq:acc-key}]
For simplicity, we drop the subscripts in this proof. 
Hence we can write the 
step 7 of \ref{alg:AICNM} as $x$ is an approximate solution of 
\begin{eqnarray}
\min_{z\in\bbR^d} \tilde{f}_{\eta}(z; y).
\end{eqnarray}
 
By the definition of $\tilde{f}(x; y)$, for any $g_{h}(x)\in \partial g(x)$, we have 
\begin{eqnarray}
\nabla \tilde{f}(x; y)=g+H(x-y)+\frac{\eta}{2}\|x-y\|(x-y)+g_{h}(x).
\end{eqnarray}
Then by the definition  $q \overset{\defi}{=} g-\nabla f(y)+\nabla F(x) - \nabla \tilde{f}_{\eta}(x; y), r \overset{\defi}{=}\|x-y\|$, then we have
\begin{eqnarray}
\|-q\|&=&\|-g+\nabla f(y)-\nabla F(x) + \nabla \tilde{f}_{\eta}(x; y)\| \nonumber\\
&=&\|\nabla f(y) - \nabla f(x) + H(x-y)+ \frac{\eta}{2}\|x-y\|(x-y)\|		\nonumber\\
&\le& \|\nabla f(y) - \nabla f(x) + \nabla^2 f(y)(x-y)\|+\|(H-\nabla^2 f(y))(x-y)\| +  \frac{\eta}{2}\|x-y\|^2 	\nonumber\\
&\le&(L_3+0.5\eta)r^2+\mu r,	\nonumber
\end{eqnarray}
which consequently implies that
\begin{eqnarray}
r\ge \frac{2\|q\|}{\mu+\sqrt{\mu^2+(4L_3+2\eta)\|q\|}}, \label{eq:acc-r}
\end{eqnarray}
due to nonnegativity of r. Then we have
\begin{eqnarray*}
(L_3 r^2 + 0.5 \mu r)^2 &\overset{\circlenum{1}}{\geq} &
(\| \nabla f(x) - \nabla f(y) - \nabla^2 f(y) (x-y)\| + \|(H-\mu I - \nabla^2 f(y)) (x-y)\|)^2\\
&\overset{\circlenum{2}}{\geq}& \| \nabla f(x) - \nabla f(y) - (H-\mu I) (x-y)\|^2\\
&\overset{\circlenum{3}}{\geq}& \|q  + (\mu + \frac{\eta}{2} \|x-y\|) (x-y) \|^2 \\
&{=}&
\|q\|^2 + (\mu + \frac{\eta}{2} r)^2 r^2 + (2\mu + \eta r)q^\top
(x-y),
\end{eqnarray*}
where \circlenum{1} is by the Lipchitz Hessian property and Assumption \ref{ass:H}, \circlenum{2} is by the triangle inequality, \circlenum{1} is by the definition of $q$.

Therefore, by the setting $\eta\ge4L_3$ and re-arranging the terms in the above relation,  
we have
\begin{eqnarray}
q^T (y-x)&\ge& \frac{1}{2\mu+\eta r}\left(\|q\|^2 + 0.75 (\mu + \frac{\eta}{2}r)^2r^2\right) \nonumber\\
&=&\frac{1}{2\mu+\eta r}\left(\|q\|^2 + 0.25 (\mu + \frac{\eta}{2}r)^2r^2\right)+\frac{2\mu + \eta r^2}{8}r^2 \nonumber\\
&\ge&0.5\|q\|r + \frac{\mu}{4}r^2 \nonumber\\
&\ge&\begin{cases}
\frac{\|q\|^2}{3\mu}+ \frac{\mu}{4}r^2, & \text{ if }  3\mu^2\ge ( 4L_3+2\eta)\|q\| \\
\sqrt{\frac{\|q\|^3}{4L_3+2\eta}} + \frac{\mu}{4}r^2,  & { \rm otherwise }
\end{cases}. \nonumber
\end{eqnarray}

Then Lemma \ref{eq:acc-key} is proved.
\end{proof}
\begin{lemma}\label{lem:recur}
Assume the constant $C_1>0, C_2>0$ in Alg. \ref{alg:AICNM} satisfies for $t\ge0$, 
\begin{eqnarray*}
C_1&\ge&\max_{t\ge0}\left\{\frac{9\mu_ta_t^2}{2A_{t+1}}-\frac{2}{3}A_t\sigma_2\right\}\\
C_2 &\ge&\max_{t\ge0}\left\{ \frac{32a_t^3L_3}{3A_{t+1}^2}-\frac{A_t\sigma_2}{R}\right\}.
\end{eqnarray*} 
For a sequence $\{G_t\}$, set $G_0=0$ and for $i\ge 1$,
\begin{eqnarray}
G_i &=& \left(\frac{A_{i+1}}{\mu_i}+\frac{9a_i^2}{2(3C_1+2A_i\sigma)}\right) \|g_i-\nabla f(y_i) - \nabla \tilde{f}_{\eta}(x_{i+1}; y_i)\|^2\nonumber\\
 &&+  \frac{9a_i^2}{2(3C_1+2A_i\sigma)}\| g_{i+1}^{\prime}- \nabla f(x_{i+1})\|^2,
\end{eqnarray}

Then if sequences $\{x_t\}$, $\{v_t\}$ are generated by Algorithm \ref{alg:AICNM}, then for all $t>0$, we have
\begin{eqnarray}
\E[A_t F(x_t)]\le \E\Big[\psi_t(v_t)+\sum_{i=0}^t G_i \Big]. \label{eq:acc-recur1}
\end{eqnarray}
\end{lemma}

\begin{proof}
Let us prove relation \ref{eq:acc-recur1} by induction over t. Since $A_0=0$, for $t=0$, we have
\begin{eqnarray}
A_0 F(x_0)=0=\min_{x\in\bbR^d}\psi_0(x)=\psi_0(v_0).
\end{eqnarray}
Assume that \ref{eq:acc-recur1} is true for some $t\ge0$. Note that for any $x\in\bbR^d$,
\begin{eqnarray*}
\psi_t(x)&=&\sum_{i=0}^{t-1}(f(x_{i+1})+\langle g^{\prime}_{i+1}, x-x_{i+1}\rangle + h(x)) +  \frac{C_1}{2}\|x-x_0\|^2+ \frac{C_2}{3}\|x-x_0\|^3\\
&=&\sum_{i=0}^{t-1}(f(x_{i+1})+\langle g^{\prime}_{i+1}, x-x_{i+1}\rangle )+A_t h(x) +  \frac{C_1}{2}\|x-x_0\|^2+ \frac{C_2}{3}\|x-x_0\|^3, \\
&\equiv&l_t(x)+A_t h(x) +  \frac{C_1}{2}\|x-x_0\|^2+ \frac{C_2}{3}\|x-x_0\|^3. 
\end{eqnarray*}

Note that $l_t(x)$ is a linear function and $A_th(x)$ is a  $A_t\sigma_2$-strongly convex function. 
By \citep[Lemma 4]{nesterov2008accelerating}, $\frac{C_2}{3}\|x-x_0\|^3$ is $\frac{C_2}{2}$-uniformly convex of degree $3$. Meanwhile $\frac{C_1}{2}\|x-x_0\|^2$ is $C_1$-strongly convex. Therefore by the optimality of $v_t$ and the induction assumption, for any $x\in\bbR^d$, we have
\begin{eqnarray*}
\E[\psi_t(x)]&\ge&\E\Big[ \psi_t(v_t) + \frac{C_1+A_t\sigma_2}{2}\|x-v_t\|^2+ \frac{C_2}{6}\|x-v_t\|^3\Big]\\
&\ge& A_t F(x_t) - \sum_{i=0}^{t}G_i+ \frac{C_1 + A_t\sigma_2 }{2}\|x-v_t\|^2+ \frac{C_2}{6}\|x-v_t\|^3\Big],
\end{eqnarray*}
where the expectation is taken on the randomness of $g^{\prime}_{i+1}$ form $i=0$ to $t-1$.

Meanwhile, by taking expectation on the randomness of $g^{\prime}_{t+1}$, we have
\begin{eqnarray*}
\E[\psi_{t+1}(v_{t+1})] &=&\E[\psi_t(v_{t+1})+a_t[f(x_{t+1})+\langle g_{t+1}^{\prime}, v_{t+1}-x_{t+1}\rangle +h(v_{t+1})]\\
 &=&\E[\psi_t(v_{t+1})+a_t[f(x_{t+1})+\langle \nabla f(x_{t+1}) , v_{t+1}-x_{t+1}\rangle\nonumber\\
 &&+\langle g_{t+1}^{\prime}- \nabla f(x_{t+1}), v_{t+1}-v_t\rangle\nonumber\\
 &&+\langle g_{t+1}^{\prime}- \nabla f(x_{t+1}), v_{t}-x_{t+1}\rangle +h(v_{t+1})]\\
  &=&\E[\psi_t(v_{t+1})+a_t(f(x_{t+1})+\langle \nabla f(x_{t+1}) , v_{t+1}-x_{t+1}\rangle\nonumber\\
 &&+\langle g_{t+1}^{\prime}- \nabla f(x_{t+1}), v_{t+1}-v_t\rangle +h(v_{t+1}))],
\end{eqnarray*}
where the last equality is by the unbiasness of $g_{t+1}^{\prime}$ that $$\E[\langle g_{t+1}^{\prime}- \nabla f(x_{t+1}), v_{t}-x_{t+1}\rangle] = \langle \E[g_{t+1}^{\prime}- \nabla f(x_{t+1})], v_{t}-x_{t+1}\rangle = 0.$$

Therefore, by taking expectation on the randomness of $g^{\prime}_{i+1}$ form $i=0$ to $t$, we have
\begin{eqnarray*}
\E[\psi_{t+1}(v_{t+1})] 
 &=& \E[\psi_t(v_{t+1})+a_t(f(x_{t+1})+\langle g_{t+1}^{\prime}, v_{t+1}-x_{t+1}\rangle +h(v_{t+1}))]\\
&=&\E[\psi_t(v_{t+1})+a_t(f(x_{t+1})+\langle \nabla f(x_{t+1}) , v_{t+1}-x_{t+1}\rangle\nonumber\\
 &&+\langle g_{t+1}^{\prime}- \nabla f(x_{t+1}), v_{t+1}-v_t\rangle +h(v_{t+1}))],\nonumber\\
 &\ge&\E\Big[A_t F(x_t)- \sum_{i=0}^{t}G_i + \frac{C_1 + A_t\sigma_2 }{2}\|v_{t+1}-v_t\|^2+ \frac{C_2}{6}\|v_{t+1}-v_t\|^3 \\
 &&+a_t(f(x_{t+1})+\langle \nabla f(x_{t+1}) , v_{t+1}-x_{t+1}\rangle\nonumber\\
 &&+\langle g_{t+1}^{\prime}- \nabla f(x_{t+1}), v_{t+1}-v_t\rangle +h(v_{t+1}))\Big]\\
  &\ge&\E\Big[A_t f(x_t)  - \sum_{i=0}^{t}G_i+ A_t h(x_t) + \frac{C_1 + A_t\sigma_2 }{2}\|v_{t+1}-v_t\|^2+ \frac{C_2}{6}\|v_{t+1}-v_t\|^3 \Big]\\
&&+a_t(f(x_{t+1})+\langle \nabla f(x_{t+1}) , v_{t+1}-x_{t+1}\rangle\nonumber\\
 &&+\langle g_{t+1}^{\prime}- \nabla f(x_{t+1}), v_{t+1}-v_t\rangle +h(v_{t+1}))\Big]
 \end{eqnarray*}
By the convexity of $f(x)$ and $h(x)$, for a $h^{\prime}(x_{t+1})\in \partial h(x_{t+1})$ such that
\begin{eqnarray*}
f(x_t)&\ge& f(x_{t+1}) + \langle \nabla f(x_{t+1}), x_t-x_{t+1}\rangle,\\
h(x_t)&\ge& h(x_{t+1}) + \langle \nabla h(x_{t+1}), x_t-x_{t+1}\rangle,\\
h(v_{t+1})&\ge& h(x_{t+1}) + \langle h^{\prime}(x_{t+1}), v_{t+1}-x_{t+1}\rangle.
\end{eqnarray*}
Substituting these inequalities above, it follows that
\begin{eqnarray*}
\E[\psi_{t+1}(v_{t+1})] &\ge&\E\Big[ A_{t+1} F(x_{t+1})  - \sum_{i=0}^{t}G_i+ \langle \nabla F(x_{t+1}), A_tx_t-A_tx_{t+1}\rangle \\
 &&+a_t \langle \nabla F(x_{t+1}), v_{t+1}-x_{t+1}\rangle+a_t \langle g_{t+1}^{\prime}- \nabla f(x_{t+1}), v_{t+1}-v_{t}\rangle \\
&& + \frac{C_1 + A_t\sigma_2 }{2}\|v_{t+1}-v_t\|^2+ \frac{C_2}{6}\|v_{t+1}-v_t\|^3\Big].
\end{eqnarray*}
Then it follows that,
\begin{eqnarray*}
\E[\psi_{t+1}(v_{t+1})]  &\overset{\circlenum{1}}{\ge}&\E\Big[A_{t+1} F(x_{t+1})  - \sum_{i=0}^{t}G_i+ A_{t+1}\langle \nabla F(x_{t+1}), y_t-x_{t+1}\rangle \\
 &&+a_t \langle \nabla F(x_{t+1}), v_{t+1}-v_t\rangle+a_t \langle g_{t+1}^{\prime}- \nabla f(x_{t+1}), v_{t+1}-v_{t}\rangle\\
 &&+ \frac{C_1 + A_t\sigma_2 }{2}\|v_{t+1}-v_t\|^2+ \frac{C_2}{6}\|v_{t+1}-v_t\|^3\Big]\\
&\overset{\circlenum{2}}{\ge}&\E\Big[A_{t+1} F(x_{t+1}) - \sum_{i=0}^{t}G_i \\
 &&+ A_{t+1}\langle  q_{t+1}, y_t-x_{t+1}\rangle - A_{t+1}\langle g_t-\nabla f(y_t) - \nabla \tilde{f}_{\eta}(x_{t+1}; y_t), y_t-x_{t+1}\rangle\\
 &&+a_t \langle  q_{t+1},  v_{t+1}-v_t\rangle  - a_t \langle  g_t-\nabla f(y_t) - \nabla \tilde{f}_{\eta}(x_{t+1}; y_t),  v_{t+1}-v_t\rangle\\
 &&+a_t \langle g_{t+1}^{\prime}- \nabla f(x_{t+1}), v_{t+1}-v_t\rangle \\
 &&+ \frac{C_1 + A_t\sigma_2 }{2}\|v_{t+1}-v_t\|^2+ \frac{C_2}{6}\|v_{t+1}-v_t\|^3\Big]\\
&\overset{\circlenum{3}}{\ge}&\E\Big[A_{t+1} F(x_{t+1}) - \sum_{i=0}^{t}G_i \\
 &&+A_{t+1}\langle q_{t+1}, y_t-x_{t+1}\rangle +a_t \langle  q_{t+1},  v_{t+1}-v_t\rangle\\
 &&-\frac{A_{t+1}}{\mu_t}\|g_t-\nabla f(y_t) - \nabla \tilde{f}_{\eta}(x_{t+1}; y_t)\|^2- \frac{A_{t+1}\mu_t}{4}\|y_t-x_{t+1}\|^2\\
 && -
\frac{a_t^2}{2C_3}\| g_t-\nabla f(y_t) - \nabla \tilde{f}_{\eta}(x_{t+1}; y_t)\|^2-\frac{C_3}{2}\|v_{t+1}-v_t\|^2
 \\
 &&-  \frac{a_t^2}{2C_4}\| g_{t+1}^{\prime}- \nabla f(x_{t+1})\|^2-\frac{C_4}{2}\|v_{t+1}-v_t\|^2 \\
 && + \frac{C_1 + A_t\sigma_2 }{2}\|v_{t+1}-v_t\|^2+ \frac{C_2}{6}\|v_{t+1}-v_t\|^3\Big],\\
 &\overset{\circlenum{4}}{=}&
 \E\Big[A_{t+1} F(x_{t+1}) - \sum_{i=0}^{t}G_i \\
 &&+A_{t+1}\langle q_{t+1}, y_t-x_{t+1}\rangle +a_t \langle  q_{t+1},  v_{t+1}-v_t\rangle- \frac{A_{t+1}\mu_t}{4}\|y_t-x_{t+1}\|^2\\
 &&-\left(\frac{A_{t+1}}{\mu_t}+\frac{a_t^2}{2C_3}\right) \|g_t-\nabla f(y_t) - \nabla \tilde{f}_{\eta}(x_{t+1}; y_t)\|^2\\
 &&-  \frac{a_t^2}{2C_4}\| g_{t+1}^{\prime}- \nabla f(x_{t+1})\|^2\\
 && +\left(\frac{C_1+\frac{2}{3} A_t\sigma_2-C_3-C_4}{2}\right)\|v_{t+1}-v_t\|^2\\
 &&+ \left(\frac{A_t\sigma_2}{6\|v_{t+1}-v_t\|} + \frac{C_2  }{6}\right)\|v_{t+1}-v_t\|^3\Big],\\
\end{eqnarray*}
where \circlenum{1} is by $y_t = (1-\alpha_t)x_t+\alpha_t v_t=\frac{A_t}{A_{t+1}}x_t+\frac{a_t}{A_{t+1}}v_t$, \circlenum{2} is by the definition of $q_{t+1}=g_t-\nabla f(y_t)+\nabla F(x_{t+1}) - \nabla \tilde{f}_{\eta}(x_{t+1}; y_t)$ in Lemma \ref{eq:acc-key} and simple arrangements,  \circlenum{3} is by triangle inequality, $C_3, C_4$ are constants, $\circlenum{4}$ is by simple rearrangements.

Then by defining $C_3 = C_4 = \frac{1}{3}\big(C_1+\frac{2}{3}A_t\sigma_2\big)$, then by the definition of $\{G_{t}\}$, we have 
\begin{eqnarray*}
\E[\psi_{t+1}(v_{t+1})] &\overset{\circlenum{1}}{=}&\E\Bigg[A_{t+1} F(x_{t+1}) - \sum_{i=0}^{t+1}G_i \\
 &&+A_{t+1}\langle q_{t+1}, y_t-x_{t+1}\rangle +a_t \langle  q_{t+1},  v_{t+1}-v_t\rangle- \frac{A_{t+1}\mu_t}{4}\|y_t-x_{t+1}\|^2\\
 && +\left(\frac{C_1+\frac{2}{3} A_t\sigma_2}{6}\right)\|v_{t+1}-v_t\|^2\\
 &&+ \left(\frac{A_t\sigma_2}{6\|v_{t+1}-v_t\|} + \frac{C_2  }{6}\right)\|v_{t+1}-v_t\|^3\Bigg],\\
 &\overset{\circlenum{2}}{\ge}&\E\Bigg[A_{t+1} F(x_{t+1}) - \sum_{i=0}^{t+1}G_i \\
 &&+A_{t+1}\min\left\{ \frac{\|q_{t+1}\|^2}{3\mu_t}  , \sqrt{\frac{\|q_{t+1}\|^3}{4L_3+2\eta}} \right\} +a_t \langle  q_{t+1},  v_{t+1}-v_t\rangle\\
 && +\left(\frac{C_1+\frac{2}{3} A_t\sigma_2}{6}\right)\|v_{t+1}-v_t\|^2\\
 &&+ \left(\frac{A_t\sigma_2}{6\|v_{t+1}-v_t\|} + \frac{C_2  }{6}\right)\|v_{t+1}-v_t\|^3\Bigg],\\
  &\overset{\circlenum{3}}{\ge}&\E\Bigg[A_{t+1} F(x_{t+1}) - \sum_{i=0}^{t+1}G_i \\
 &&+\min\left\{ \frac{\|q_{t+1}\|^2}{3\mu_t}  , \sqrt{\frac{\|q_{t+1}\|^3}{4L_3+2\eta}} \right\} +a_t \langle  q_{t+1},  v_{t+1}-v_t\rangle\\
 && +\left(\frac{C_1+\frac{2}{3} A_t\sigma_2}{6}\right)\|v_{t+1}-v_t\|^2\\
 &&+ \left(\frac{A_t\sigma_2}{6R} + \frac{C_2  }{6}\right)\|v_{t+1}-v_t\|^3\Bigg],\\
 \end{eqnarray*}
where \circlenum{1} is by the definition of $G_{t+1}$, \circlenum{2} is by Lemma \ref{eq:acc-key}, \circlenum{3} is by the assumption that $\|v_{t+1}-v_t\|\le R$.

Then the following task is to prove the error term
\begin{eqnarray}
&&A_{t+1}\min\left\{ \frac{\|q_{t+1}\|^2}{3\mu_t}  , \sqrt{\frac{\|q_{t+1}\|^3}{4L_3+2\eta}} \right\} +a_t \langle  q_{t+1},  v_{t+1}-v_t\rangle +\left(\frac{C_1+\frac{2}{3} A_t\sigma_2}{6}\right)\|v_{t+1}-v_t\|^2\nonumber\\
&&\quad\quad+ \left(\frac{A_t\sigma_2}{6R} + \frac{C_2  }{6}\right)\|v_{t+1}-v_t\|^3\ge 0 \label{eq:error}
\end{eqnarray}

By Lemma \ref{lem:inner-prod}, we have
\begin{eqnarray}
a_t \langle q_{t+1}, v_{t+1}-v_t\rangle \ge - \frac{a_t^2}{2C_5} \|q_{t+1}\|^2 - \frac{C_5}{2}\|v_{t+1}-v_t\|^2,
\end{eqnarray}
and 
\begin{eqnarray}
a_t \langle q_{t+1}, v_{t+1}-v_t\rangle \ge - \frac{C_6}{3}\|v_{t+1}-v_t\|^3-\frac{2a_t^{3/2}}{3}\left(\frac{1}{C_6}\right)^{1/2}\|q_{t+1}\|^{3/2},
\end{eqnarray}
where $C_5, C_6$ are  constants.

Therefore by setting 
\begin{eqnarray}
C_5 = \frac{C_1+\frac{2}{3}A_t\sigma_2}{3} \label{eq:1111}\\
C_6 = \frac{A_t\sigma_2 +C_2 R}{2R} \label{eq:1112}
\end{eqnarray}
and assuming that 
\begin{eqnarray}
A_{t+1}\sqrt{\frac{1}{4L_3+2\eta}}- \frac{2a_t^{3/2}}{3}\left(\frac{1}{C_6}\right)^{1/2}\ge0 \label{eq:1113}\\
\frac{A_{t+1}}{3\mu_t}-\frac{a_t^2}{2C_5}\ge 0  \label{eq:1114}
\end{eqnarray}

Then the error term \eqref{eq:error} will be $\ge0$.

The conditions in \eqref{eq:1111}-\eqref{eq:1114} is equivalent to
\begin{eqnarray*}
C_1&\ge&\frac{9\mu_ta_t^2}{2A_{t+1}}-\frac{2}{3}A_t\sigma_2\\
C_2 &\ge& \frac{16a_t^3(2L_3+\eta)}{9A_{t+1}^2}-\frac{A_t\sigma_2}{R}\\
 &=& \frac{32a_t^3L_3}{3A_{t+1}^2}-\frac{A_t\sigma_2}{R}\\
\end{eqnarray*}

Lemma \ref{lem:recur} is proved.

\end{proof}

\begin{proof}[Proof of Theorem \ref{thm:prime-result}]
By Lemmas \ref{lem:111} and \ref{lem:recur}, and the optimality of $v_t$, by setting $x=x^*$ in Lemma \ref{lem:111}, we have
\begin{eqnarray}
\E[A_t F(x_t)]&\le& \E\Big[\psi_t(v_t)+\sum_{i=0}^t G_i \Big]\nonumber\\
&\le& \E\Big[\psi_t(x^*)+\sum_{i=0}^t G_i \Big]\nonumber\\
&\le& A_t F(x^*) + \frac{C_1}{2}\|x^*-x_0\|^2+\frac{C_2}{3}\|x^*-x_0\|^3+\sum_{i=0}^t G_i
\nonumber\\
&=& A_t F(x^*) + \frac{C_1}{2}\|x^*-x_0\|^2+\frac{C_2}{3}\|x^*-x_0\|^3+\sum_{i=1}^t G_i\nonumber
\end{eqnarray}
Dividing $A_t$ on both sides, we get Theorem \ref{thm:prime-result}.
\end{proof}

\begin{proof}[Proof of Corollary \ref{thm:acc-online}]
By the setting in Theorems \ref{thm:acc-convex} and \ref{thm:acc-strong-convex} and using a similar analysis to the analysis in the proof of
Corollary \ref{thm:online}, we can prove Corollary \ref{thm:acc-online}. 
\end{proof}

\begin{proof}[Proof of Corollary \ref{thm:acc-finite}]
In the finite-sum setting, the analysis is similar to the online stochastic setting. The main difference is that the samples of gradient and Hessian can be larger than the number $n$ of the terms in the sum. Therefore, after the number of the samples evaluated by Lemmas \ref{lem:vector}  and \ref{lem:matrix} is larger than $n$, then we use the exact gradient in the following iteration. By this setting, we can know that in the convex setting, we only need $\min\{\tilde{\mathcal{O}}(\epsilon^{-2}), \tilde{\mathcal{O}}(n\epsilon^{-1/3})\}$   in the convex setting and $\min\left\{\tilde{\mathcal{O}}(\sigma_2^{-2/3}\epsilon^{-1}),\tilde{\mathcal{O}}(\sigma_2^{-1/3}n)\right\} $ in the strongly convex setting.
\end{proof}

\section{Proof of Section \ref{sec:cubic-svrg}}

Compared with Prox-SVRG in \cite{xiao2014proximal}, the difference is only the number of the inner iteration $M_s$ and the learning rate for each outer iteration $\tau_s$. Therefore the conclusions of a fixed stage that only use the smoothness property in \cite{xiao2014proximal} can still hold for Alg. \ref{alg:cubic}. 
By the proof of \citep[Theorem 1]{xiao2014proximal}, we have Lemma \ref{lem:basic}.
\begin{lemma}[\cite{xiao2014proximal}]\label{lem:basic}
In the $s$-th outer iteration, if we have the following inequality. 
\begin{eqnarray}
\E [P(\tilde{w}_{s}) - P(w^*)]\!\!\!\!\!&\le&\!\!\!\!\! \frac{1}{2\tau_s(1-4L_{2}\tau_s)M_s} \|\tilde{w}_{s-1}-w^*\|_2^2 + \frac{4L_{2}\tau_s(M_s+1)}{(1-4L_2 \tau_s)M_s}(P(\tilde{w}_{s-1}) - P(w^*)),\label{eq:cubic-1}
\end{eqnarray}
where 
 the expectation is taken on the randomness of the current iteration (the randomness of all the previous iteration is fixed).
\end{lemma}

\begin{proof}[Proof of Theorem \ref{thm:converge}]
If $P(\tilde{w}_{s-1})-P(w^*)\ge \frac{1}{m^3}$, then by the steps 5-6, we have $M_s=100\kappa m$ and $\tau_s = \tau_0m^{-\frac{1}{2}}(P(\tilde{w})-P(w^*))^{-\frac{1}{3}}$. 
Then it follows that
\begin{eqnarray}
&&\E[P(\tilde{w}_s)-P(w^*)]\nonumber\\
&\overset{\circlenum{0}}{\le}& \frac{1}{2\tau_s(1-4L_{2}\tau_s)M_s} \|\tilde{w}_{s-1}-w^*\|_2^2 \nonumber\\
&&\quad+ \frac{4L_{2}\tau_s(M_s+1)}{(1-4L_2 \tau_s)M_s}(P(\tilde{w}_{s-1}) - P(w^*))\nonumber\\
&\overset{\circlenum{1}}{\le}&\frac{1}{2\tau_s(1-4L_{2}\tau_s)M_s} \left(\frac{6}{\sigma_3}\right)^{2/3}(P(\tilde{w}_{s-1}) - P(w^*))^{2/3}  \nonumber\\
&&\quad+ \frac{4L_{2}\tau_s(M_s+1)}{(1-4L_2 \tau_s)M_s}(P(\tilde{w}_{s-1}) - P(w^*))\nonumber\\
&\overset{\circlenum{2}}{=}& \Bigg( \frac{1}{200\kappa\tau_s(1-4L_{2}\tau_s)m(P(\tilde{w}_{s-1}) - P(w^*))^{1/3} } \left(\frac{6}{\sigma_3}\right)^{2/3}  \nonumber\\
&&\quad+ \frac{4L_{2}\tau_s(M_s+1)}{(1-4L_2 \tau_s)M_s}\Bigg)(P(\tilde{w}_{s-1}) - P(w^*))\nonumber\\
&\overset{\circlenum{3}}{=}&\Bigg( \frac{1}{100L_2\tau_0(1-4L_{2}\tau_s)m^{1/2}(P(\tilde{w}_{s-1}) - P(w^*))^{1/6} } \nonumber\\
&&\quad+ \frac{4L_{2}\tau_0(M_s+1)(P(\tilde{w}_{s-1}) - P(w^*))^{-1/6} }{m^{1/2}(1-4L_2 \tau_s)M_s}\Bigg)(P(\tilde{w}_{s-1}) - P(w^*))\nonumber\\
&{=}&\frac{1}{\sqrt{m}}\left(\frac{1}{100L_2\tau_0(1-4L_2\tau_s)}+\frac{4L_2\tau_0(M+1)}{(1-4L_2\tau_s)M_s}\right)(P(\tilde{w}_{s-1}) - P(w^*))^{5/6}\nonumber\\
&\overset{\circlenum{4}}{=}& \frac{1}{\sqrt{m}}\left(\frac{1}{100L_2\tau_0(1-4L_2\tau_0)} +\frac{4L_2\tau_0(100\kappa m+1)}{100(1-4L_2\tau_0)\kappa m}\right)(P(\tilde{w}_{s-1}) - P(w^*))^{5/6}\nonumber\\
&\overset{\circlenum{5}}{=}& \frac{\rho}{\sqrt{m}}(P(\tilde{w}_{s-1}) - P(w_*))^{5/6},\nonumber
\end{eqnarray}
where \circlenum{0} is by Lemma \ref{lem:basic}, \circlenum{1} is by uniform convexity in Assumption \ref{ass:sigma3-smooth}, \circlenum{2} is by the value of $M_s$,  \circlenum{3} is by the definition of $\kappa$ and the value of $\eta_s$, \circlenum{4} is by the value of $M_s$, and \circlenum{4} is by the definition of $\rho$.

Because $x^{5/6}$ is a concave function,  taking expectation on the randomness on all the history, we have 
\begin{eqnarray}
\E\left[P(\tilde{w}_{s}) - P(w_*)\right]&\le& \frac{\rho}{\sqrt{m}}\E[(P(\tilde{w}_{s-1}) - P(w_*))^{5/6}] \nonumber\\
&\le& \frac{\rho}{\sqrt{m}}(\E[P(\tilde{w}_{s-1}) - P(w_*)])^{5/6}.\label{eq:proof-case1-2}
\end{eqnarray}
Telescoping \eqref{eq:proof-case1-2}, we obtain the result of the case $s\le s_1$. 

When $s>s_1$, $i.e.,$
 $P(\tilde{w}_{s-1})-P(w_*)\le \frac{1}{m^3}$, which is equivalent to
 \begin{eqnarray}
 \rho(P(\tilde{w}_{s-1}) -P(w_*))\le \frac{\rho}{\sqrt{m}}[P(\tilde{w}_{s-1})-P(w_*)]^{\frac{5}{6}},\label{eq: case1-3}
 \end{eqnarray}
 then let Alg. \ref{alg:cubic} converge with an exponential rate $\rho$ will be faster. By the setting of $M_s$ and $\tau_s$, and use a similar analysis, we have 
$$\E[P(\tilde{w}_{s-1}) -P(w_*)]\le\rho(P(\tilde{w}_{s-1}) -P(w_*)).$$
 Telescoping the above inequality, we obtain the result of the case $s>s_1$. 

Theorem \ref{thm:converge} is proved.

\end{proof}

\end{document}